\documentclass[10pt]{article}
\usepackage{latexsym,amsfonts,amssymb,amsmath,amsthm}

\usepackage[usenames,dvipsnames]{color}

\begin{document}
\setlength{\baselineskip}{14pt}

\parindent 0.5cm
\evensidemargin 0cm \oddsidemargin 0cm \topmargin 0cm \textheight 22cm \textwidth 16cm \footskip 2cm \headsep
0cm

\newtheorem{theorem}{Theorem}[section]
\newtheorem{lemma}{Lemma}[section]
\newtheorem{proposition}{Proposition}[section]
\newtheorem{definition}{Definition}[section]
\newtheorem{example}{Example}[section]
\newtheorem{corollary}{Corollary}[section]

\newtheorem{remark}{Remark}[section]

\numberwithin{equation}{section}

\def\p{\partial}
\def\I{\textit}
\def\R{\mathbb R}
\def\C{\mathbb C}
\def\l{\lambda}
\def\a{\alpha}
\def\O{\Omega}
\def\e{\epsilon}
\def\ls{\lambda^*}
\def\D{\displaystyle}
\def\wyx{ \frac{w(y,t)}{w(x,t)}}
\def\imp{\Rightarrow}
\def\tE{\tilde E}
\def\tX{\tilde X}
\def\tH{\tilde H}
\def\tu{\tilde u}
\def\d{\mathcal D}
\def\aa{\mathcal A}
\def\DH{\mathcal D(\tH)}
\def\bE{\bar E}
\def\bH{\bar H}
\def\M{\mathcal M}
\renewcommand{\labelenumi}{(\arabic{enumi})}

\def\disp{\displaystyle}
\def\undertex#1{$\underline{\hbox{#1}}$}
\def\card{\mathop{\hbox{card}}}
\def\sgn{\mathop{\hbox{sgn}}}
\def\exp{\mathop{\hbox{exp}}}
\def\OFP{(\Omega,{\cal F},\PP)}
\newcommand\JM{Mierczy\'nski}
\newcommand\RR{\ensuremath{\mathbb{R}}}
\newcommand\CC{\ensuremath{\mathbb{C}}}
\newcommand\QQ{\ensuremath{\mathbb{Q}}}
\newcommand\ZZ{\ensuremath{\mathbb{Z}}}
\newcommand\NN{\ensuremath{\mathbb{N}}}
\newcommand\PP{\ensuremath{\mathbb{P}}}
\newcommand\abs[1]{\ensuremath{\lvert#1\rvert}}

\newcommand\normf[1]{\ensuremath{\lVert#1\rVert_{f}}}
\newcommand\normfRb[1]{\ensuremath{\lVert#1\rVert_{f,R_b}}}
\newcommand\normfRbone[1]{\ensuremath{\lVert#1\rVert_{f, R_{b_1}}}}
\newcommand\normfRbtwo[1]{\ensuremath{\lVert#1\rVert_{f,R_{b_2}}}}
\newcommand\normtwo[1]{\ensuremath{\lVert#1\rVert_{2}}}
\newcommand\norminfty[1]{\ensuremath{\lVert#1\rVert_{\infty}}}

\title{Approximations of Random Dispersal Operators/Equations by Nonlocal Dispersal Operators/Equations \thanks{Partially supported by NSF grant DMS--0907752}}

%
%
%
%

\author{
Wenxian Shen  \\
Department of Mathematics and Statistics\\
Auburn University\\
Auburn, AL 36849, U.S.A. \\\\
Xiaoxia Xie\\
Department of Applied  Mathematics\\
Illinois Institute of Technology\\
Chicago, IL 60616, U.S.A.
}

\date{}
\maketitle

\noindent {\bf Abstract.}
This paper is concerned with the approximations of random dispersal operators/equations by nonlocal
dispersal operators/equations.  It first proves that the solutions of properly rescaled nonlocal dispersal initial-boundary
value problems converge to the solutions of the corresponding random dispersal initial-boundary value problems. Next,
it proves that the principal spectrum points of nonlocal dispersal operators with properly rescaled kernels converge
to the principal eigenvalues of the corresponding random dispersal operators.
Finally, it proves that the unique positive time periodic solutions of nonlocal dispersal  KPP equations with properly
 rescaled kernels converge to the unique positive time periodic solutions of the corresponding random dispersal KPP equations.

\bigskip

\bigskip


\newpage

\section{Introduction}
\setcounter{equation}{0}

Both random dispersal evolution equations (or reaction diffusion equations) and nonlocal dispersal
evolution equations (or differential integral equations) are widely used to model diffusive systems in applied
sciences. Random dispersal equations  of the form
\begin{equation}
\label{random-eq}
\begin{cases}
\p_t u(t, x)=\Delta u(t, x)+F(t,x,u), \quad &x\in D,\cr
B_{r, b} u(t, x)=0,&x\in \p D\,\, (x\in \RR^N\,\, {\rm if}\,\, D=\RR^N),
\end{cases}
\end{equation}
 are usually used to model
 diffusive systems which exhibit local internal interactions (i.e.  the movements of organisms in the systems occur randomly
  between adjacent
spatial locations)  and have been extensively studied (see
\cite{Aro, ArWe1, ArWe2, CaCo, Fif1, Fisher, HessWein, Kolmo, Mur,
Ske, Zha2}, etc.). In \eqref{random-eq},  the domain $D$ is either a bounded smooth domain in $\RR^N$ or $D=\RR^N$. When $D$ is a bounded domain,
 either $B_{r, b}u=B_{r,D}u:=u$ (in such case, $B_{r, D} u=0$ on $\p D$
represents homogeneous Dirichlet boundary condition),  or $B_{r, b}
u=B_{r,N}u:=\frac{\p u}{\p {\bf n}}$ (in such case, $B_{r, N}u=0$
on $\p D$ represents homogeneous Neumann boundary condition),
and when $D=\R^N$, it is assumed that $F(t,x,u)$ is periodic in
$x_j$ with period $p_j$ and $B_{r, b} u=B_{r,P}u:=u(t,x+p_j{\bf
e_j})-u(t,x)$ with ${\bf
e_j}=(\delta_{1j},\delta_{2j},\cdots,\delta_{Nj})$ ($\delta_{ij}=0$
if $i\not =j$ and $\delta_{ij}=1$ if $i=j$) (in such case, $B_{r,
P}u=0$ in $\RR^N$ represents  periodic boundary condition).

Many applied systems exhibit nonlocal internal interaction  (i.e.  the movements of organisms in the systems occur  between non-adjacent spatial locations). Nonlocal dispersal evolution equations
 of the form
 \begin{equation}
 \label{nonlocal-eq}
 \begin{cases}
 \p_t u(t, x)=\nu \int_{ D\cup D_b} k(y-x)[u(t,y)-u(t, x)]dy+F(t,x,u),\quad &x\in \bar D,\cr
 B_{n, b}u(t, x)=0, &x\in D_b \,\, {\rm if}\,\, D_b\not =\emptyset,
 \end{cases}
 \end{equation}
 are often used to model diffusive systems which exhibit nonlocal internal
 interactions and have been recently studied by many people
 (see \cite{BaZh,  CaCoLoRy, ChChRo, CoDaMa, Cov, CoDu, CoDaMa1, Fif, GrHiHuMiVi, HMMV, KaLoSh, LiLoWa, LiSuWa, ZhLiWa}, etc.).
 In \eqref{nonlocal-eq}, $D$ is either a smooth bounded domain of $\RR^N$ or $D=\RR^N$;
 $\nu$ is the dispersal rate;  the kernel function $k(\cdot)$ is a smooth and nonnegative  function   with compact support (the size of the support reflects the dispersal distance) and
  $\int_{\R^N}k(z)dz=1$.
 When $D$ is bounded, either $D_b=D_D:=\RR^N\backslash{\bar D}$
   and $B_{n, b}u=B_{n,D}:=u$ (in such case, $u=0$ on $\RR^N\backslash \bar D$ represents homogeneous Dirichlet type boundary condition), or
   $D_b=D_N:=\emptyset$  (in such case,  nonlocal diffusion takes place only in $\bar D$
  and hence $D_N=\emptyset$ represents homogeneous Neumann  type
 boundary condition);
 when $D=\RR^N$, it is assumed that $F(t,x+p_j{\bf e_j},u)=F(t,x,u)$, $D_b=D_P:=\RR^N$, and $B_{n, b} u=B_{n,P}u:= u(t,x+p_j{\bf e_j})-u(t,x)$ (hence $B_{n, P} u=0$ on $\RR^N$ represents  periodic boundary condition).

Observe that \eqref{nonlocal-eq} with  $D_b=D_D$ and
$B_{n, b} u=B_{n,D}u$ can be rewritten as
\begin{equation}
\label{nonlocal-eq-d}
\p_t u(t, x)=\nu\left[\int_D k(y-x)u(t, y)dy-u(t, x)\right]+F(t, x, u), \quad x\in \bar D;
\end{equation}
that \eqref{nonlocal-eq} with $D_b=D_N$ reduces to
\begin{equation}
\label{nonlocal-eq-n}
\p_t u(t, x)=\nu\int_D k(y-x)\left[u(t, y)-u(t, x)\right]dy+F(t, x, u), \quad x\in\bar D;
\end{equation}
and  that \eqref{nonlocal-eq} with $D=D_P$, $F(t,x,u)$ being periodic
in $x_j$ with period $p_j$, and $B_{n, b} u=B_{n,P}u$ can be written as
\begin{equation}
\label{nonlocal-eq-p}
\begin{cases}
\p_t u(t, x)=\nu\int_{\R^N} k(y-x)\left[u(t, y)-u(t, x)\right]dy+F(t, x, u), \quad & x\in \R^N,\\
u(t, x)=u(t, x+p_j{\bf e_j}), \quad &x\in \R^N
\end{cases}
\end{equation}
$(j=1, 2, \cdots N)$.

A huge amount of research has been carried out toward various dynamical aspects of random dispersal evolution equations of the form
\eqref{random-eq}. There are also many research works toward various dynamical aspects of nonlocal dispersal evolution equations
of the form \eqref{nonlocal-eq}. It has been seen that
 random dispersal evolution equations  with Dirichlet, or  Neumann, or period boundary condition and nonlocal dispersal
  evolution equations with the corresponding boundary condition share many similar properties.
  For example, a comparison principle holds for both equations. There are also many differences between these two types of
  dispersal evolution equations. For example, solutions of random dispersal evolution equations have smoothness and certain
  compactness properties,   but solutions of nonlocal dispersal evolution equations do not have such properties.
 Nevertheless, it is  expected that nonlocal dispersal evolution equations with Dirichlet, or Neumann, or periodic
 boundary condition and small dispersal distance possess similar dynamical behaviors as those of random dispersal evolution  equations
 with the  corresponding boundary condition and that certain dynamics of random dispersal evolution  equations with
 Dirichlet, or Neumann, or periodic boundary condition can be approximated by the dynamics of nonlocal dispersal evolution equations
 with the corresponding boundary condition and properly rescaled kernels.
  It is of great theoretical and practical importance to investigate whether such naturally expected properties actually hold or not.

 The objective of the current paper is  to investigate how the dynamics of random dispersal operators/equations
 can be approximated by those of  nonlocal dispersal
 operators/equations from three different perspectives, that is,  from initial-boundary value problem point of view, from spectral problem point of view, and  from asymptotic behavior point of view.
 To this end, we assume that $k(\cdot)$ is of the form,
\begin{equation}
 \label{kernel-function}
 k(z)=k_{\delta}(z):= \frac1{\delta^N}k_0\left(\frac{z}{\delta}\right)
 \end{equation}
 for some $k_0(\cdot)$ satisfying that  $k_0(\cdot)$ is a smooth, nonnegative, and symmetric
 (in the sense that $k_0(z)=k_0(z')$ whenever $|z|=|z'|$) function supported on the unit ball $B(0,1)$ and
  $\int_{\R^N}k_0(z)dz=1$, where $\delta (>0)$ is called the dispersal distance. We also assume that
\begin{equation}
\label{nu-delta-eq}
  \nu=\nu_{\delta}:=\frac {C}{\delta^2},
  \end{equation}
   where $C=\Big(\frac 12 \int_{\RR^N} k_0(z)z_N^2dz\Big)^{-1}$.  Throughout the rest of this paper,
   we will distinguish the three boundary conditions by
$i=1, 2, 3$. Let
\begin{equation*}
\label{x-space} X_1=X_2=\{u(\cdot)\in C(\bar D, \R)\}
\end{equation*}
with $\|u\|_{X_i}=\max_{x\in\bar D}|u(x)| (i=1, 2)$,
\[
X_3=\{u\in C(\R^N, \R)| u(x+p_j{\bf e_j})=u(x)\},
\]
with $\|u\|_{X_3}=\max_{x\in \R^N}|u(x)|$. Let
\begin{equation*}
\label{x-positive-cone} X_i^+=\{u\in X_i\,|\, u(x)\geq 0\}
\end{equation*}
($i=1, 2, 3$).  For $u^1(x), u^2(x)\in X_i$, we define
\[
u^1\leq u^2 (u^1\geq  u^2) \text{ if } u^2-u^1\in  X_i^+ (u^1-u^2\in
X_i^+)
\]
(i=1, 2, 3). Note that $X_1=X_2$ and the introduction of
$X_2$ is for convenience.

First, we investigate the approximations of solutions to the initial-boundary value problem associated to \eqref{random-eq}, that is,
\begin{equation}
\label{main-random-general}
\begin{cases}
\p_t u(t, x)=\Delta u+F(t,x,u),\quad & x\in D,\cr B_{r,b}(t, x) u=0, \quad &
x\in\p D \quad (x\in \R^N \text{ if } D=\R^N), \cr u(s,x)=u_0(x),\quad &
x\in \bar D
\end{cases}
\end{equation}
by solutions to the initial-boundary value problem associated to \eqref{nonlocal-eq} with $k(\cdot)=k_\delta(\cdot)$ and $\nu=\nu_\delta$, that is,
\begin{equation}
\label{main-nonlocal-general}
\begin{cases}
\p_t u(t, x)=\nu_{\delta}\int_{D\cup D_b}
k_\delta(y-x)[u(t,y)-u(t,x)]dy+F(t,x,u),\,\ & x\in \bar D,\cr
B_{n,b}u(t, x)=0,\quad & x\in D_b \,\,\, {\rm if}\,\, D_b\not =\emptyset, \cr
u(s,x)=u_0(x),\quad & x\in \bar D,
\end{cases}
\end{equation}
where $B_{r,b}=B_{r,D}$ (resp. $B_{n,b}=B_{n,D}$ and $D_b=D_D$), or $B_{r,b}=B_{r,N}$ (resp.  $D_b=D_N(=\emptyset$)), or
$B_{r,b}=B_{r,P}$ (resp. $B_{n,b}=B_{n,P}$ and $D_b=D_P$). In the rest of this paper, we assume

\medskip

\noindent {\bf (H0)} {\it $D\subset \R^N$ is either a bounded $C^{2+\alpha}$ domain for some
$0<\alpha<1$ or $D=\RR^N$; $k_\delta(\cdot)$ is as in \eqref{kernel-function} and
$\nu_\delta$ is as in \eqref{nu-delta-eq}; $F(t,x,u)$ is $C^1$ in $t\in\RR$ and
$C^3$ in
$(x,u)\in \RR^N\times\RR$, and when $D=\RR^N$, $F$ is periodic in $x_j$ with period
$p_j$, that is, $F(t,x+p_j{\bf e_j},u)=F(t,x,u)$ for
$j=1,2,\cdots,N$.}

\medskip

 Note that, by general semigroup theory (see \cite{Henry, Paz}),  for any $s\in\RR$ and any $u_0\in X_i\cap C^1(\bar D)$  with
 $B_{r, b}u_0=0$ on $\p D$,
\eqref{main-random-general} with $b=D$ if $i=1$, $b=N$ if $i=2$, and $b=P$ if $i=3$ has a unique (local) solution, denoted by $u_i(t,x;s,u_0)$. Similarly,  for any  $s\in\RR$ and any $u_0\in X_i$,
\eqref{main-nonlocal-general} with $b=D$ if $i=1$, $b=N$ if $i=2$, and $b=P$ if $i=3$
 has a
unique (local) solution,  denoted  by $u_i^\delta(t,x;s,u_0)$.

Among others, we prove

\medskip

\noindent {\bf Theorem A.} {\it  Assume that  for  given $1\le i\le 3$,  $\delta_0>0$, $s\in\RR$, $T>0$,  and $u_0\in X_i\cap C^{3}(\bar  D)$ with $B_{r, b}u_0=0$ if
$D$ is bounded ($b=D$ if $i=1$ and $b=N$ if $i=2$),  $u_i(t, x; s, u_0)$ and $u_i^\delta(t, x; s, u_0)$ exist on $[s, s+T]$ for all $0<\delta\le \delta_0$.
Assume also that $\sup_{s\le t\le s+T,x\in \bar D, 0<\delta\leq\delta_0} |u_i(t,x;s,u_0)|<\infty$. Then,
$$
\lim_{\delta\to 0}\sup_{t\in[s, s+T]}\|u_i^{\delta}(t, \cdot;s, u_0)-u_i(t, \cdot;
s,u_0)\|_{X_i}=0.
$$ }

\medskip

It should be pointed out that Theorem A is the basis for the study of approximations of various dynamics of
random dispersal evolution equations by those of nonlocal dispersal evolution equations. It should also be pointed out that when $F(t, x, u)\equiv 0$ in \eqref{main-random-general} and \eqref{main-nonlocal-general}, similar results to Theorem A have been proved in \cite{CoElRo} and \cite{CoElRoWo} for the  Dirichlet and Neumann boundary condition cases, respectively.

Secondly, we investigate the principal eigenvalues of  time periodic  random dispersal eigenvalue problems of the form
\begin{equation}
\label{main-random-eigenvalue}
\begin{cases}
-\p_t u+\Delta u+a(t,x)u=\lambda u,\quad& x\in D,\cr
B_{r,b}u=0,\quad &x\in\p D \,\ (x\in \R^N \text{ if } D=\R^N),\cr
u(t+T,x)=u(t,x),\quad &x\in D,
\end{cases}
\end{equation}
 and  their  nonlocal counterparts  of the form
\begin{equation}
 \label{main-nonlocal-eigenvalue}
\begin{cases}
-\p_t u+\nu_{\delta}\int_{D\cup D_b} k_\delta(y-x)\left[u(t,y)-u(t,x)\right]dy+a(t,x)u=\lambda u,\quad& x\in \bar D,\cr
B_{n,b}u=0,\quad &x\in D_b   \,\, {\rm if}\,\, D_b\not=\emptyset, \cr
u(t+T, x)=u(t, x),\quad &x\in \bar D,
\end{cases}
\end{equation}
where  $a(t+T,x)=a(t,x)$, and when $D=\RR^N$, $a(t+T ,x+p_j{\bf e_j})=a(t,x)$
 for $j=1,2,\cdots,N$, and $B_{r,b}=B_{r,D}$ (resp. $B_{n,b}=B_{n,D}$ and $D_b=D_D$), or $B_{r,b}=B_{r,N}$ (resp. $D_b=D_N(=\emptyset)$) or
$B_{r,b}=B_{r,P}$ (resp. $B_{n,b}=B_{n,P}$ and $D_b=D_P$). We assume that $a(t,x)$ is a $C^1$ function in $(t,x)\in\RR\times\RR^N$.

The eigenvalue problems of \eqref{main-random-eigenvalue}, in particular, their associated principal eigenvalue problems, are
extensively studied and quite well understood (see \cite{Dan, Daners, DanersA, PHess, HeShZh, HuShVi, Nad2,  ShVi}, etc.).
For example, with any one of the three boundary conditions, it is known that the largest real part, denoted by $\lambda^r(a)$, of the spectrum set of \eqref{main-random-eigenvalue} is an isolated algebraically simple eigenvalue  with a positive eigenfunction, and for any other $\lambda$ in the spectrum set of \eqref{main-random-eigenvalue}, $\text{Re}\lambda\le \lambda^r(a)$ ($\lambda^r(a)$ is called the
 {\it principal eigenvalue} of $\eqref{main-random-eigenvalue}$ in literature).

The eigenvalue problems \eqref{main-nonlocal-eigenvalue} have also been studied recently by many people (see \cite{BaZh-T, Cov, HuShVi, RaSh, ShVi, ShZh0, ShZh1, ShXi}, etc.). Let $\lambda^{\delta}(a)$ be the largest real part of the spectrum set of \eqref{main-nonlocal-eigenvalue} with any one of the three boundary conditions. $\lambda^{\delta}(a)$ is called the
 {\it principal spectrum point} of \eqref{main-nonlocal-eigenvalue}.  $\lambda^{\delta}(a)$ is also called the
 {\it principal eigenvalue} of \eqref{main-nonlocal-eigenvalue}, if it is an isolated algebraically simple eigenvalue  with a positive eigenfunction (see Definition \ref{n-pev-def} for detail).  Note that $\lambda^\delta(a)$ may not be an eigenvalue
 of \eqref{main-nonlocal-eigenvalue} (see \cite{Cov}, \cite{ShZh0} for examples). Hence the principal eigenvalue of \eqref{main-nonlocal-eigenvalue} may not exist.
In \cite{ShXi}, the authors of the current paper studied the dependence of principal spectrum points or principal eigenvalues (if exist) of nonlocal dispersal operators on underlying parameters ($\delta, a(\cdot)$, and $\nu$) in a spatially heterogeneous but temporally homogeneous case. However, the understanding is still little to many interesting questions regarding the principal spectrum points or principal eigenvalues (if exist) of \eqref{main-nonlocal-eigenvalue}.  In this paper, we show that the principal eigenvalue of \eqref{main-random-eigenvalue} can be approximated by the principal spectrum point of \eqref{main-nonlocal-eigenvalue}. In fact, we show

\medskip

 \noindent{\bf Theorem B.}
 {\it
 $\lim_{\delta \to 0}\lambda^{\delta}(a)=\lambda^r(a)$.
 }
\medskip

We remark that Theorem B is another basis for the study of
approximations of various dynamics of random dispersal evolution
equations by those of nonlocal dispersal evolution equations. We also
remark that some necessary and sufficient conditions are provided in
\cite{RaSh} and \cite{RaShZh}   for $\lambda_\delta(a)$ to be the principal eigenvalue of
\eqref{main-nonlocal-eigenvalue}. Among other, it is proved in
\cite[Theorem A]{RaSh} and \cite[Theorem 3.1]{RaShZh}  that $\lambda^{\delta}(a)$ is the
principal eigenvalue of \eqref{main-nonlocal-eigenvalue} if and only
if
\[
\lambda^{\delta}(a)>\max_{x\in \bar D}\left\{-\frac C{\delta^2} +\frac1T\int_0^Ta(t, x)dt\right\}.
\]
This together with Theorem B implies the following remark.
 \begin{remark}
\label{necessary-sufficient-lambda}
$\lambda^{\delta}(a)$ is the principal eigenvalue of  \eqref{main-nonlocal-eigenvalue}, provided $\delta\ll 1$.
\end{remark}

 Thirdly, we  explore the asymptotic dynamics of the following time periodic dispersal evolution equations,
 \begin{equation}
 \label{main-random-kpp}
 \begin{cases}
 \p_t u=\Delta u+uf(t,x,u),\quad & x\in D,\cr
 B_{r,b}u=0,\quad & x\in\p D \,\ (x\in \R^N \text{ if } D=\R^N),
 \end{cases}
 \end{equation}
 and
 \begin{equation}
 \label{main-nonlocal-kpp}
\begin{cases}
 \p_t u=\nu_{\delta}\int_{D\cup D_b} k_\delta(y-x)[u(t,y)-u(t,x)]dy+uf(t,x,u),\quad & x\in \bar D,\cr
B_{n,b}u=0,& x\in D_b\,\,\, {\rm if}\,\, D_b\not=\emptyset,
\end{cases}
 \end{equation}
 where $D$ is as in {\bf (H0)}. In the rest of this paper, we
 assume that

 \medskip

 \noindent {\bf (H1)} {\it  $f$ is $C^1$ in $t\in\RR$ and  $C^3$ in $(x,u)\in \RR^N\times\RR$; $f(t, x, u)<0$ for $u\gg 1$ and $\partial _u f(t, x, u)<0$ for $u\geq0$};
 $f(t+T,x,u)=f(t,x,u)$; and when $D=\RR^N$, $f(t+T, x,u)=f(t, x+p_j{\bf e_j}, u)=f(t,x,u)$ for $j=1, 2, \cdots, N$.

\medskip

\noindent {\bf (H2)} {\it For \eqref{main-random-kpp}, $\lambda^r(f(\cdot, \cdot, 0))>0$, where $\lambda^r(f(\cdot, \cdot, 0))$ is the principle eigenvalue of
\eqref{main-random-eigenvalue} with $a(t, x)=f(t, x, 0)$.
}

\medskip

\noindent {\bf (H2)${_\delta}$} {\it   For \eqref{main-nonlocal-kpp}, $\lambda^{\delta}(f(\cdot, \cdot, 0))>0$, where $\lambda^{\delta}(f(\cdot, \cdot, 0))$ is
the principle spectrum  point of \eqref{main-nonlocal-eigenvalue} with $a(t, x)=f(t, x, 0)$.
}

\medskip

Equations \eqref{main-random-kpp} and \eqref{main-nonlocal-kpp} are widely used to model population dynamics of species exhibiting random interactions and nonlocal interactions,
 respectively (see \cite{BaZh, CoDaMa1, Nad1}, etc.  for \eqref{main-random-kpp} and \cite{RaSh} for \eqref{main-nonlocal-kpp}). Thanks to the pioneering works of
  Fisher \cite{Fisher} and Kolmogorov et al. \cite{Kolmo} on the following special case of \eqref{main-random-kpp},
\[
\p_t u=u_{xx}+u(1-u), \quad x\in \R,
\]
\eqref{main-random-kpp} and \eqref{main-nonlocal-kpp} are referred to as Fisher type or KPP type equations.

The dynamics of  \eqref{main-random-kpp} and
\eqref{main-nonlocal-kpp} have been studied in many papers (see \cite{HessWein, Nad1, Zhaox} and references therein for \eqref{main-random-kpp}, and
\cite{RaSh} and references therein for
\eqref{main-nonlocal-kpp}).  With conditions (H1) and (H2), it is
proved that  \eqref{main-random-kpp}  has exactly two  nonnegative
time periodic solutions, one is $u\equiv 0$ which is unstable and
the other one,
 denoted by $u^*(t,x)$, is   asymptotically stable and strictly positive (see \cite[Theorem 3.1] { Zhaox}, see also \cite[Theorems 1.1, 1.3]{Nad1}).
  Similar results for \eqref{main-nonlocal-kpp} under the assumptions (H1) and (H2)$_{\delta}$ are proved in \cite[Theorem E]{RaSh}.
 We denote the strictly positive time periodic solution of \eqref{main-nonlocal-kpp} by $u_\delta^*(t,x)$.

Note that, by Theorem B and Remark 1.1,  (H2) implies  (H2)${_\delta}$  when $0<\delta\ll1$. Hence, we only assume (H2) in the following theorem.  In this paper, we show that

%
%

\medskip

\noindent {\bf Theorem C.}
{\it
If (H1) and (H2) hold, then for any $\epsilon>0$, there exists $\delta_0>0$, such that for all $0<\delta<\delta_0$, we have
$$\sup_{t\in [0, T]} \|u_{\delta}^*(t, \cdot)-u^*(t, \cdot)\|_{C(\bar D, \R)}\leq \epsilon.
$$
}
\medskip

 Theorems A-C in the above show that many important
dynamics of random dispersal equations can be approximated by the
corresponding dynamics of nonlocal dispersal equations, which is of
both great theoretical and practical importance.

The rest of the paper is organized as follows. In section 2, we
explore the approximation of solutions of random dispersal evolution
equations by the solutions of nonlocal dispersal evolution equations
and prove Theorem A. In section 3, we investigate the
approximation of principal eigenvalues of time periodic random
dispersal operators by the principal spectrum points of time periodic nonlocal
dispersal operators and prove Theorem B. We study  in section 4 the
approximation of the asymptotic dynamics of time periodic KPP
equations with random dispersal by the asymptotic dynamics of time
periodic KPP equations with nonlocal dispersal  and prove Theorem C.

\section{Approximation of Initial-boundary Value Problems of Random Dispersal Equations by Nonlocal Dispersal Equations}

In this section, we explore the approximation of solutions to
\eqref{main-random-general} by the solutions to
\eqref{main-nonlocal-general}. We first  present some comparison principle for \eqref{main-random-general} and
\eqref{main-nonlocal-general}. Then we prove Theorem A. Though the ideas of the
proofs of Theorem A for different types of boundary conditions are the same,
different techniques are needed for different boundary conditions. We  hence give proofs of Theorem A for different
boundary conditions in different subsections.

\subsection{Comparison principle for random and nonlocal
dispersal evolution equations}

In this subsection, we present a comparison
principle for  random and nonlocal evolution equations, which will be applied  in the proof of Theorem A in this section
as well as in the proofs of Theorem B and C in sections 3 and 4.

\begin{definition}[Super- and sub- solutions]
A continuous function $u(t, x)$ on $[s, s+T)\times \RR^N$ is called a
super-solution (sub-solution) of \eqref{main-nonlocal-general} on $(s,s+T)$ if
for any $x\in \bar D$, $u(t, x)$ is differentiable on $(s, s+T)$ and
satisfies that
\begin{equation*}
\begin{cases}
\p_t u(t, x)\geq(\le)\nu_{\delta}\int_{D\cup D_b} k_{\delta}(y-x)[u(t,y)-u(t,x)]dy
+F(t, x, u), \quad &x\in\bar D,\\
B_{n,b}u(t, x)\geq(\le) 0, &x\in D_b\,\, {\rm if}\,\, D_b\not =\emptyset,\\
u(s,x)\geq(\le) u_0(x), & x\in \bar D,
\end{cases}
\end{equation*}
when $b=D$ or $N$, or that
\begin{equation*}
\begin{cases}
\p_t u(t, x)\geq(\le)\nu_{\delta}\int_{\RR^N} k_{\delta}(y-x)[u(t,y)-u(t,x)]dy
+F(t, x, u), \quad &x\in \R^N,\\
B_{n,b}u(t, x)= 0, &x\in \R^N,\\
u(s,x)\geq(\le) u_0(x), & x\in\R^N,
\end{cases}
\end{equation*}
when $b=P$.

Super-solutions and sub-solutions of \eqref{main-random-general} on $(s,s+T)$ are defined in an analogous way.
\end{definition}

\begin{proposition}[Comparison principle]
\label{comparison-prop}
\begin{itemize}
\item[(1)] Suppose that $u^-(t,x)$ and $u^+(t,x)$ are sub-solution and super-solution of \eqref{main-random-general} on $(s,s+T)$, respectively, then
$$
u^-(t,x)\le u^+(t,x)\quad \forall \,\, t\in [s,s+T),\,\, x\in\bar D.
$$

\item[(2)] Suppose that $u^-(t,x)$ and $u^+(t,x)$ are sub-solution and super-solution of \eqref{main-nonlocal-general} on $(s,s+T)$,  respectively,  then
$$
u^-(t,x)\le u^+(t,x)\quad \forall \,\, t\in [s,s+T),\,\, x\in\bar D.
$$
\end{itemize}
\end{proposition}

\begin{proof} (1) It follows from comparison principle for parabolic equations.

(2) It follows from \cite[Proposition 3.1]{RaSh}.
\end{proof}

\subsection{Proof of Theorem A in the Dirichlet boundary condition
case}

In this subsection, we prove Theorem A in the Dirichlet boundary
case. Throughout this subsection, we assume (H0), and
$B_{r,b}u=B_{r,D}u$ in \eqref{main-random-general},  and
$D_b=D_D(=\RR^N\backslash \bar D)$ and $B_{n,b} u=B_{n,D}u$ in
\eqref{main-nonlocal-general}. Note that $D\cup D_b=\RR^N$ in this case. Without loss of generality, we assume
$s=0$.

\begin{proof} [Proof of Theorem A in the Dirichlet boundary condition case]
Let $u_0\in C^3(\bar D)$ with $u_0(x)=0$ for $x\in\p D$.
 Let $u_1^{\delta}(t,x)$ be the solution of \eqref{main-nonlocal-general} with $s=0$ and
 $u_1(t,x)$ be the solution of \eqref{main-random-general} with $s=0$.
 Suppose that $u_1(t,x)$ and $u_1^\delta(t,x)$ exist on $[0,T]$. By regularity of solutions for parabolic equations,
  $u_1\in C^{1+\frac{\alpha}{2},2+\alpha}((0, T]\times \bar D)\cap C^{0,2+\alpha}( [0,T]\times \bar D)$.
Let $\tilde u_1$ be an extension of $u_1$ to $[0,T]\times \RR^N$ satisfying that $\tilde u_1\in C^{0,2+\alpha}([0,T]\times \RR^N)$. Define
\begin{equation*}
 L_{\delta}(z)(t,x)=\nu_\delta\int_{\RR^N} k_{\delta} (y-x) [z(t,
y)-z(t, x)]dy.
\end{equation*}
Let  $G(t,x)=\tilde u_1(t,x)$ for $(t,x)\in [0,T]\times \RR^N\backslash \bar D$.
Then $\tilde u_1$ verifies
\begin{equation*}
\begin{cases}
\p_t\tilde u_1(t, x)= L_{\delta}(\tilde u_1)(t,x) + F_{\delta} (t,x) + F(t,x, \tilde u_1(t,x)), \quad& x\in \bar D, \,\, \,\ \quad \,  t\in (0, T],\\
\tilde u_1(t,x)=G(t,x),\quad \,\ &x\in \RR^N\backslash \bar D, t\in [0, T],\\
\tilde u_1(0, x)=u_0(x), \quad & x\in \bar D,
\end{cases}
\end{equation*}
where
\begin{align*}
F_{\delta}(t,x)&=\Delta \tilde u_1(t,x)-{L}_{\delta}(\tilde u_1)(t,x)\\
&=\Delta \tilde u_1(t,x)-\nu_\delta\int_{\RR^N}k_\delta(y-x)(\tilde u_1(t,y)-\tilde u_1(t,x))dy.
\end{align*}

 Let  $w_1^{\delta}=\tilde
u_1-u_1^{\delta}$. We then have
\begin{equation}
\label{thma-d-eq0}
\begin{cases}
\p_t w_1^{\delta}(t, x)= L_{\delta}(w_1^{\delta})(t,x)+F_{\delta}(t,x)+ a_1^\delta(t,x)w_1^{\delta}(t,x), \quad & x\in \bar D, \,\,\, \quad \,\  t\in (0, T],\\
w_1^\delta(t,x)=G(t,x),\quad &x\in \RR^N\backslash \bar D,\,  t\in [0,T],\\
w_1^{\delta}(0, x) =0, \quad & x\in \bar D,
\end{cases}
\end{equation}
where $a_1^\delta(t,x)=\int_0^1 F_u[t, x,
u_1^\delta(t,x)+\theta(\tilde u_1(t,x)-u_1^{\delta}(t,x))]d {\theta}$.

We claim  that
\begin{equation}
\label{thma-d-eq1}
\begin{cases}
\sup_{t\in [0,
T]}\|F_{\delta}(t,\cdot)\|_{X_1}=O({\delta}^{\alpha}),\cr \sup_{t\in
[0, T], x\in \RR^N\setminus\bar D,{\rm dist}(x,\p D)\le
\delta}|G(t,x)|=O(\delta).
\end{cases}
\end{equation}
In fact,
\begin{align*}
&\Delta \tilde u_1(t,x)-\nu_\delta\int_{\RR^N} k_{\delta}(y-x)(\tilde u_1(t,y)-\tilde u_1(t,x))dy\\
&=\Delta \tilde u_1(t,x)-\nu_\delta\int_{\RR^N }\frac{1}{\delta^N}k_0\left(\frac{y-x}{\delta}\right)(\tilde u_1(t,y)-\tilde u_1(t,x))dy\\
&= \Delta \tilde u_1(t,x)-\nu_\delta\int_{\RR^N}k_0(z)(\tilde u_1(t,x+\delta z)-\tilde u_1(t,x))dz\\
&= \Delta \tilde u_1(t,x)-\nu_\delta\int_{\RR^N}k_0(z)\left[\frac{\delta^2 z_N^2}{2!}\Delta \tilde u_1(t,x)+O(\delta^{2+\alpha})\right]dz\\
&= \Delta \tilde u_1(t,x)-\left[\nu_\delta\delta^2\int_{\RR^N}k_0(z)\frac{ z_N^2}{2}dz\right]\Delta \tilde u_1(t,x)+O(\delta^{\alpha})\\
&= \Delta \tilde u_1(t,x)-\Delta \tilde u_1(t,x)+O(\delta^{\alpha})\\
&=O(\delta^{\alpha})\quad \forall \,\, x\in\bar D,
\end{align*}
and
\begin{align*}
|G(t,x)|&=|\tilde u_1(t,x)|\\
&\le \sup_{t\in [0, T], x\in \RR^N\setminus D, z\in\p D, \text{dist}(x,z)\le \delta}|\tilde u_1(t,x)-u_1(t,z)|\\
&=O(\delta)\quad \forall\,\, x\in \RR^N\backslash \bar D,\,\,{\rm dist}(x,\p D)\le\delta.
\end{align*}
Therefore, \eqref{thma-d-eq1} holds.

Next, let $\bar w$ be given by
\[
\bar w(t,x) =e^{At}(K_1{\delta}^{\alpha}t)+K_2 \delta,
\]
where $A=\underset{t\in [0, T],x\in\bar D,0<\delta\leq\delta_0}{\max}a_1^\delta(t,x)$. By direct calculation,  we have
\begin{equation}
\label{thma-d-eq2}
\begin{cases}
\p_t\bar w(t,x)=L_{\delta}(\bar w)+a_1^\delta(t,x)\bar w +\bar F_\delta(t,x)\quad &  x\in\bar D,\quad \quad \,\,\ \,\, t\in (0,T],\\
\bar w(t,x)=e^{At}(K_1{\delta}^{\alpha}t)+K_2 \delta,\quad & x\in \RR^N\backslash\bar D,\quad t\in [0,T],\\
\bar w(0,x)=K_2\delta,\quad & x\in\bar D,
\end{cases}
\end{equation}
where
$$
\bar F_\delta(t,x)=
e^{At}K_1{\delta}^{\alpha}+[A-a_1^\delta(t,x)]e^{At} K_1{\delta}^{\alpha}t-a_1^\delta(t,x)K_2\delta.
$$
By \eqref{thma-d-eq1}, there are  $0<\tilde \delta_0\le \delta_0$ and $K_1,K_2>0$ such that
\begin{equation}
\label{thma-d-eq3}
\begin{cases}
F_\delta(t,x)\le \bar F_\delta(t,x),\quad & x\in\bar D,\,\,\,\, t\in
[0,T],\cr G(t,x)\le e^{At}(K_1{\delta}^{\alpha}t)+K_2 \delta,\quad&
x\in \RR^N\backslash \bar D, \,\, {\rm dist}(x,\p D)\le \delta,\, t\in [0,T],
\end{cases}
\end{equation}
when $0<\delta<\tilde \delta_0$. By  \eqref{thma-d-eq0},
\eqref{thma-d-eq2}, \eqref{thma-d-eq3}, and Proposition
\ref{comparison-prop},  we obtain
\begin{equation}
\label{thma-d-eq4}
 w^{\delta}(t,x)\le \bar w(t,x)=e^{At}(K_1{\delta}^{\alpha}t)+K_2\delta\quad \forall\, x\in\bar D,\,\, t\in [0,T]
\end{equation}
for $0<\delta<\tilde \delta_0$.

Similarly,  let $\underline w(t,x)=e^{At}(-K_1{\delta}^{\alpha}t)-K_2\delta$. We can prove that
for $0<\delta<\tilde \delta_0$  (by reducing $\tilde\delta_0$ if necessary),
\begin{equation}
\label{thma-d-eq5} w^{\delta}(t,x)\geq \underline
w(t,x)=-e^{At}(K_1{\delta}^{\alpha}t)-K_2\delta\quad \forall\,\,
x\in\bar D,\,\, t\in [0,T].
\end{equation}
By \eqref{thma-d-eq4} and \eqref{thma-d-eq5} we have
\[
|w^{\delta}(t,x)|\leq e^{At} K_1{\delta}^{\alpha}t+K_2\delta\quad
\forall\,\, x\in\bar D,\,\, t\in [0,T],
\]
which implies that there is $C(T)>0$ such that for any $0<\delta<\tilde \delta_0$,
\[
\sup_{t\in [0, T]}\|u_1(\cdot,t)-u_1^{\delta}(\cdot,t)\|_{X_1}\leq C(T){\delta}^{\alpha}.
\]
Theorem A in the Dirichlet boundary condition case then follows.
\end{proof}

\begin{remark}
If the homogeneous Dirichlet boundary conditions $B_{r,D}u=u=0$ on $\p D$ and $B_{n, D}u=u=0$ on $\RR^N\backslash\bar D$ are
changed to nonhomogeneous Dirichlet boundary conditions $B_{r, D} u=u=g(t,x)$ on $\p D$ and $B_{n, D} u=u=g(t,x)$ on $\RR^N\backslash \bar D$,
Theorem A also holds, which can be proved by the similar arguments as above.
\end{remark}

\subsection{Proof of Theorem A in the Neumann boundary condition
case}

 In this subsection, we prove Theorem A  in the Neumann boundary condition
case. Throughout this subsection, we assume (H0), and
$B_{r,b}u=B_{r,N}u$ in \eqref{main-random-general}, and
 $D_b=D_N=\emptyset$  in
\eqref{main-nonlocal-general}. Without loss of generality, we assume
$s=0$.

We first introduce two lemmas. To this end, for given $\delta>0$ and $d_0>0$, let $D_{\delta}=\{z\in D|\mathrm{dist}(z,
\partial D)<d_0\delta\}$.

\begin{lemma}
\label{thma-n-lm1}
 Let $\theta\in C^{1+\frac \alpha 2,2+\alpha}( (0,T]\times \times\RR^N)\cap C^{0,2+\alpha}( [0,T]\times\RR^N)$
 and $\frac{\partial \theta}{\partial {\bf n}}=h$ on $\partial D$, then for $x\in D_{\delta}$ and $\delta$ small,
\begin{align*}
&\frac1{\delta^2}{\int}_{\RR^N\backslash D} k_{\delta}(y-x)(\theta(t, y)-\theta(t,x))dy\\
&=\frac1{\delta}{\int}_{\RR^N\backslash D} k_{\delta}(y-x){\bf n}(\bar x)\cdot\frac{y-x}{\delta}h(\bar x, t)dy\\
&+{\int}_{\RR^N\backslash D} k_{\delta}(y-x)\sum_{|\beta|=2}\frac{D^{\beta}\theta}{2}(\bar x, t)\left[\left(\frac{y-\bar x}{\delta}\right)^{\beta}-\left(\frac{x-\bar x}{\delta}\right)^{\beta}\right]dy+O(\delta^\alpha),
\end{align*}
where $\bar x$ is the orthogonal projection of $x$ on the boundary of $D$ so that $\|\bar x-y\|\leq 2d_0\delta$ and
${\bf n}(\bar x)$ is the exterior unit normal vector of $\p D$ at $\bar x$.
\end{lemma}

\begin{proof}
See \cite[Lemma 3]{CoElRo}.
\end{proof}

\begin{lemma}
\label{thma-n-lm2}
There exist $K>0$ and $\bar \delta >0$ such that for $\delta<\bar \delta$,
\[
\underset{\RR^N\backslash D}{\int} k_{\delta}(y-x){\bf n}(\bar x)\frac{y-x}{\delta}dy \geq K \underset{\RR^N\backslash D}{\int} k_{\delta}(y-x)dy.
\]
\end{lemma}

\begin{proof}
See \cite[Lemma 4]{CoElRo}.
\end{proof}

\begin{proof}[Proof of Theorem A in the Neumann boundary  condition case]
Suppose that $u_0\in C^3(\bar D)$. Let $u_2^{\delta}(t,x)$ be the
solution to \eqref{main-nonlocal-general} with $s=0$  and $u_2(t,x)$
be the solution to \eqref{main-random-general} with $s=0$. Assume
that $u_2(t,x)$ and $u_2^\delta(t,x)$ exist on $[0,T]$. Then $u_2\in
C^{ 1+\frac{\alpha}{2},2+\alpha}((0, T]\times \bar D)$. Let $\tilde
u_2$ be an extension of $u_2$ to $[0, T]\times\RR^N$ satisfying that
$\tilde u_2\in C^{1+\frac{\alpha}{2},2+\alpha}(
(0,T]\times \RR^N)\cap C^{0,2+\alpha}([0,T]\times \RR^N)$. Define
\[
L_{\delta}(z)(t,x)=\nu_\delta\int_{D} k_{\delta}(y-x)(z(t,y)-z(t,x))dy,
\]
and
\[
\tilde L_{\delta}(z)(t,x)=\nu_\delta \int_{\RR^N} k_{\delta}(y-x)(z(t,y)-z(t,x))dy.
\]
Set $w_2^{\delta}=u_2^{\delta}-\tilde u_2$. Then
\begin{align*}
\p_t w_2^{\delta}(t, x) &=\p_t u_2^{\delta}(t, x)-\p_t\tilde u_2(t, x)\notag\\
& =[L_{\delta}(u_2^{\delta})(t, x)+F(t,x, u_2^{\delta})]-[\Delta \tilde u_2(t, x)+F(t,x, \tilde u_2)]\notag\\
&=L_{\delta}(w_2^{\delta})(t, x)+ a_2^\delta(t,x) w_2^{\delta}(t,x)+F_\delta(t,x),
\end{align*}
where $ a_2^\delta(t,x)=\int_0^1F_u(t, x, \tilde u_2(t,x)+\theta (u^\delta_2(t,x)-\tilde u_2(t,x)))d\theta$ and
\[
F_{\delta}(t,x)=\tilde L_{\delta}(\tilde u_2)(t,x)-\Delta \tilde u_2(t,x)
-\nu_\delta \int_{\RR^N\backslash D} k_{\delta}(y-x)(\tilde u_2(t,y)-\tilde u_2(t,x))dy.
\]
Hence $w_2^{\delta}$ verifies
\begin{equation}
\label{thma-n-eq1}
\begin{cases}
\p_t w_2^{\delta}(t, x)=L_{\delta}(w_2^{\delta})(t, x)+a_2^\delta(t,x)w_2^{\delta}(t, x)+F_{\delta}(t,x),&\quad x\in\bar D,\\
w_2^{\delta}(0,x)=0, &\quad x\in\bar D.
\end{cases}
\end{equation}

To prove the theorem, let us pick an auxiliary function $v$ as a solution to
\begin{equation*}
\begin{cases}
\p_t v(t, x)=\Delta v(t, x)+a_2^\delta(t,x) v(t,x)+h(t,x), \quad & x\in D, \,\,\,\,t\in (0, T],\\
\frac{\partial v}{\partial {\bf n}} (t, x)=g(t,x), \quad & x\in \partial D, \,t\in [0, T],\\
v(0,x)=v_0(x),\quad & x\in D
\end{cases}
\end{equation*}
for some smooth functions $h(t,x)\geq 1$, $g(t,x)\geq 1$ and $v_0(x)\geq 0$ such that $v(t,x)$ has an extension $\tilde v(t,x)\in C^{ 1+\frac{\alpha}{2},
2+\alpha}( (0, T]\times\RR^N)\cap C^{0,2+\alpha}([0,T]\times\RR^N)$.
Then $v$ is a solution to
\begin{equation}
\label{thma-n-eq1-2}
\begin{cases}
\p_t v(t, x)=L_{\delta}(v)(t, x)+a_2^\delta(t,x)v(t, x)+H(t,x,\delta), \quad & x\in \bar D, t\in (0, T],\\
v(0,x)=v_0(x), \quad &x\in \bar D, t\in [0, T],
\end{cases}
\end{equation}
where
$$
H(t,x,\delta)=\Delta \tilde v(t,x)-\tilde L_{\delta} (v)(t,x)+\nu_\delta\int_{\RR^N\backslash D} k_{\delta}(y-x) (\tilde v(t, y)-\tilde v(t,x))dy+h(t,x).
$$
By Lemma \ref{thma-n-lm1} and the first estimate in \eqref{thma-d-eq1}, we have the following estimate for $H(x, t, \delta)$:
\begin{align}
\label{thma-n-eq2}
H(t,x,\delta)&=\Delta \tilde v(t,x)-\tilde L_{\delta}  (v)(t,x)+\frac{C}{\delta^2}\int_{\RR^N\backslash D} k_{\delta}(y-x)  (\tilde v(t, y)-\tilde v(t,x))dy+h(t,x)\notag\\
&\geq \frac{C}{\delta} \underset {\RR^N\backslash D}{\int}  k_{\delta}(y-x){\bf n}(\bar x)\frac{y-x}{\delta} g(\bar x, t)dy \notag\\
&\qquad +C\underset {\RR^N\backslash D}{\int} k_{\delta}(y-x)  \sum_{|\beta|=2}\frac{D^{\beta}\tilde v}{2} (\bar x, t)\left[\left(\frac{y-\bar x}{\delta}\right)^{\beta}-\left(\frac{x-\bar x}{\delta}\right)^{\beta}\right]dy +1-C_1\delta^{\alpha}\notag\\
&\geq \frac{C}{\delta} g(\bar x, t) \underset{\RR^N\backslash D}{\int}  k_{\delta}(y-x)  {\bf n}(\bar x) \frac{y-x}{\delta}dy-D_1C\underset {\RR^N\backslash D}{\int} k_{\delta}(y-x) dy+\frac 12
\end{align}
for some constants $D_1$ and $C_1$ and $\delta$ sufficiently small such that $C_1\delta^{\alpha}\leq \frac 12$. Then  Lemma
 \ref{thma-n-lm2} implies that there exist  $C'>0$ and $\delta'$ such that
\[
\frac1{\delta} \underset{\RR^N\backslash D}{\int} k_{\delta}(y-x) {\bf n}(\bar x)\frac{y-x}{\delta}dy\geq \frac{ C'}{\delta}\underset{\RR^N\backslash D}{\int} k_{\delta}(y-x) dy,
\]
if $\delta <\delta'$. This implies that
\begin{align}
\label{thma-n-eq3}
H(x, t, \delta) \geq \left[\frac{CC'g(\bar x, t)}{\delta}-D_1\right]\underset{\RR^N\backslash D}{\int} k_{\delta}(y-x) dy+\frac 12,
\end{align}
 if $\delta<\delta'$.

We now estimate  $F_\delta(t, x)$.
By Lemmas \ref{thma-n-lm1}, \ref{thma-n-lm2}, the first estimate in \eqref{thma-d-eq1}, and the fact that $\frac{\partial \tilde u_2}{\partial {\bf n}}=0$ on $\p D$, we have
\begin{align}
\label{thma-n-eq4}
F_\delta(t,x) & =O(\delta^\alpha)+\nu_\delta\underset{\RR^N\backslash D}{\int} k_{\delta}(y-x) (\tilde u_2(t, y)-\tilde u_2(t,x))dy\nonumber\\
& =O(\delta^\alpha) +C\underset{\RR^N\backslash D}{\int} k_{\delta}(y-x) \sum_{|\beta|=2}\frac{D^{\beta}\theta}{2}(\bar x, t)\left[ \left(\frac{y-\bar x}{\delta}\right)^{\beta}-\left(\frac{x-\bar x}{\delta}\right)^{\beta}\right]dy\nonumber\\
&\leq C_2 \delta^\alpha+D_1C\underset{\RR^N\backslash D}{\int} k_{\delta}(y-x) dy\nonumber\\
&=C_2\delta^\alpha +D_2\underset{\RR^N\backslash D}{\int} k_{\delta}(y-x) dy
\end{align}
for some $C_2>0$ and $D_2>0$.
Given $\epsilon >0$, let $v_\epsilon=\epsilon v$. By \eqref{thma-n-eq1-2}, $v_{\epsilon}$ satisfies
\begin{equation}
\label{thma-n-eq1-1}
\begin{cases}
\p_t v_{\epsilon}(t, x)-L_{\delta}(v_{\epsilon})(t, x)-a_2^\delta(t,x) v_\epsilon(t, x)=\epsilon H(t,x, \delta), &\quad x\in\bar D,\\
v_{\epsilon}(0,x)=\epsilon v_0(x),&\quad x\in\bar D.
\end{cases}
\end{equation}
By \eqref{thma-n-eq3} and \eqref{thma-n-eq4}, there exist $C_3>0$ and $0<\tilde \delta_0<\delta_0$  such that for $0<\delta\le \tilde \delta_0$,
\begin{align}
\label{thma-n-eq5}
F_{\delta}(t,x)\leq C\delta^\alpha +D_2\underset{\RR^N\backslash D}{\int}k_{\delta}(y-x) dy\leq \frac{\epsilon}{ 2} +\frac{C_3\epsilon}{\delta}
 \underset{\RR^N\backslash D}{\int} k_{\delta}(y-x) dy=\epsilon H(x, t, \delta)\quad \forall \, x\in\bar D,\,\, t\in [0,T].
\end{align}
Then by \eqref{thma-n-eq1}, \eqref{thma-n-eq1-1}, \eqref{thma-n-eq5},  and Proposition \ref{comparison-prop}, we have
\[
-M\epsilon\leq -v_{\epsilon}\leq w_2^{\delta}\leq v_{\epsilon}\leq M\epsilon \quad \forall \,\,  \delta\leq \tilde \delta_0,
\]
where $\displaystyle M=\max_{t\in [0,T],x\in\bar D}v(t,x)$. This implies
\[
\sup_{t\in[0, T]}\|u_2^{\delta}(t, \cdot)-u_2(t, \cdot)\|_{X_2}\rightarrow 0,
\quad \mathrm{as} \,\ \delta \rightarrow 0.
\]
Theorem A in the Neumann boundary condition is thus proved.
\end{proof}

\subsection{Proof of Theorem A in the periodic boundary condition
case}

 In this subsection,  we prove Theorem A in  the periodic  boundary condition
case. Throughout this subsection, we assume (H0),  $B_{r,b}u=B_{r,P}u$ in
\eqref{main-random-general}, and $B_{n,b}u=B_{n,P}u$ in
\eqref{main-nonlocal-general}. Without loss of generality again, we
assume $s=0$.

\begin{proof}[Proof of Theorem A in the periodic boundary case]
Suppose that $u_0\in X_3\cap C^3(\RR^N)$.
Let $u_3^\delta(t,x)$ be the solution to
\eqref{main-nonlocal-general} with $s=0$ and $u_3(t,x)$  be the
solution to \eqref{main-random-general} with $s=0$. Suppose that
$u_3(t,x)$ and $u_3^{\delta}(t, x)$ exist on $[0,T]$.  Set $w_3^\delta=u_3^\delta-u_3$. Then
$w_3^\delta$ satisfies
\begin{equation}
\begin{cases}
\p_t w_3^\delta(t, x)=\nu_\delta\int_{\RR^N} k_{\delta}(y-x) (w_3^\delta(t, y)-w_3^\delta(t,x))dy+a_3^\delta(t,x)w_3^\delta(t, x)+F_\delta(t,x), & x\in\R^N,\, t\in (0, T],\\
w_3^\delta(t,x)=w_3^\delta (t, x+p_j\bold e_j), &x\in \R^N, \, t\in [0, T],\\
w_3^\delta (0, x)=0, &x\in \R^N,
\end{cases}
\end{equation}
where $a_3^\delta(t, x)=\int_0^1F_u(t, x,   u_3(t,x)+\theta (u^\delta_3(t,x)-u_3(t,x)))d\theta$ and $F_\delta(t,x)=\nu_\delta\int_{\RR^N} k_\delta(y-x) [u_3(t, y)-u_3(t,x)]dy-\Delta u_3$.
Let \[
\bar w(t,x) =e^{At}(K_1{\delta}^{\alpha}t)+K_2 \delta,
\]
where $\displaystyle A=\max_{ t\in [0, T], x\in \R^N, 0<\delta\leq\delta_0}a_3^\delta(t, x)$. Applying the similar approach as in the Dirichlet boundary condition case, we can show that there are $K_1>0$, $K_2>0$, and
$\delta_0>0$ such that for $0<\delta<\delta_0$,
$$
-\bar w(t,x)\le w_3^\delta(t,x)\le \bar w(t,x)\quad \forall\,\, x\in\R^N,\,\, t\in [0,T].
$$
Theorem A in the periodic boundary condition case then follows.
\end{proof}

\section{Approximation of Principal Eigenvalues of Time Periodic Random Dispersal Operators by Nonlocal Dispersal Operators}

In this section, we investigate the approximation of principal
eigenvalues of time periodic random dispersal operators by the
principal spectrum points of time periodic nonlocal dispersal
operators. We first recall some basic properties of principal
eigenvalues of time periodic random dispersal or parabolic operators,
and basic properties of principal spectrum points of time periodic
nonlocal dispersal operators. We then prove Theorem B.

\subsection{Basic properties}

In this subsection, we present basic properties of principal
eigenvalues of time periodic  parabolic operators and basic
properties of principal spectrum points of time periodic nonlocal
dispersal operators.

 Let
\begin{equation*}
\label{x-T-space} \mathcal X_1=\mathcal X_2=\{ u\in C(\R\times \bar
D, \R) | u(t+T, x)=u(t, x)\}
\end{equation*}
with norm $\|u\|_{\mathcal X_i}=\sup_{t\in [0, T]}\|u(t, \cdot)\|_{X_i}
(i=1, 2)$,
\[
\mathcal X_3=\{u\in C(\R\times \R^N, \R)| u(t+T,x)=u(t, x+p_j{\bf
e_j })=u(t, x)\}
\]
with norm $\|u\|_{\mathcal X_3}=\sup_{t\in [0, T]}\|u(t, \cdot)\|_{X_3}$,
and
\begin{equation*}
\label{x-T-positive-cone} \mathcal X_i^+=\{u\in \mathcal X_i| u(t,
x)\geq 0\}
\end{equation*}
$(i=1, 2, 3)$. And for $u^1, u^2\in \mathcal X_i$, we
define
\[
u^1\leq u^2 (u^1\geq u^2) \text{ if } u^2-u^1\in \mathcal X_i^+ \,
(u_1-u_2\in \mathcal X_i^+)
\]
$(i=1, 2, 3)$.
For given $a(\cdot,\cdot)\in \mathcal X_i\cap C^1(\RR\times\RR^N)$ , let $L^\delta_i(a):
\mathcal{D}(L^\delta_i(a))\subset \mathcal X_i\to \mathcal X_i$ be
defined as follows,
\begin{equation}
\label{l-D-op}
 (L^\delta_1(a)u)(t,x)=-\p_tu(t, x) +\nu_{\delta}\left[\int_{D}k_\delta(y-x)u(t,
y)dy-u(t, x)\right]+a(t, x)u(t, x),\quad (t,x)\in\RR\times\bar D,
\end{equation}
\begin{equation}
\label{l-N-op} (L^\delta_2(a)u)(t,x)=-\p_tu(t, x)
+\nu_{\delta}\int_{D}k_\delta(y-x)[u(t, y)-u(t, x)]dy+a(t, x)u(t,
x),\quad (t,x)\in\RR\times\bar D,
\end{equation}
and
\begin{equation}
\label{l-P-op} (L^\delta_3(a)u)(t, x)=-\p_tu(t, x)
+\nu_{\delta}\int_{\R^N}k_\delta(y-x)[u(t, y)-u(t, x)]dy+a(t, x)u(t,
x),\quad (t,x)\in \RR\times\RR^N.
\end{equation}

We first recall the definition of principal spectrum
points/eigenvalues of time periodic nonlocal dispersal operators.

\begin{definition}
\label{n-pev-def} Let
\[
\lambda^{\delta}_i(a)=\sup\{{\rm Re}\lambda|\lambda\in
\sigma(L^\delta_i( a))\}
\]
for $i=1,2,3$.
\begin{itemize}
\item[(1)] $\lambda^\delta_i(a)$  is called the  {\rm principal spectrum point} of
$L^\delta_i( a)$.

\item[(2)]   If $\lambda^{\delta}_i(a)$ is an isolated algebraically simple eigenvalue of
$L^\delta_i( a)$ with a positive eigenfunction, then
$\lambda^{\delta}_i(a)$ is called the {\rm principal eigenvalue} of
$L^\delta_i( a)$ or it is said that  $L^\delta_i( a)$ {\rm has a
principal eigenvalue}.
\end{itemize}
\end{definition}

 For the time periodic random dispersal operators, let $a(\cdot,\cdot)\in
\mathcal X_i\cap C^1(\RR\times\RR^N)$, and $L_i(a):\mathcal D(L_i(a))\subset\mathcal X_i\to
\mathcal X_i$ be defined as follows,
\[
(L_i(a)u)(t, x)=-\p_t u(t, x)+\Delta u(t, x)+a(t, x)u(t, x)
\]
for $i=1,2,3$. Note that for $u\in\mathcal{D}(L_1(a))$,
$B_{r,D}u=u=0$ on $\p D$ and for $u\in\mathcal{D}(L_2(a))$,
$B_{r,N}u=\frac{\p u}{\p {\bf n}}=0$ on $\p D$. Let
\[
\lambda^r_i(a)=\sup\{{\rm Re}\lambda|\lambda\in \sigma(L_i(a))\}.
\]
It is well known that
  $\lambda^r_i(a)$ is an isolated algebraically simple eigenvalue of
$L_i( a)$ with a positive eigenfunction  (see \cite{PHess}) and  $\lambda^r_i(a)$ is called
the {\it principal eigenvalue} of $L_i( a)$.

Next we  derive some properties of the principal spectrum points of
nonlocal dispersal operators by using the spectral radius of the
solution  operators of the associated evolution equations. To this
end, for $i=1, 2, 3$,  define $\Phi_i^{\delta}(t, s;  a): X_i\to
X_i$ by
\[
(\Phi_i^{\delta}(t, s;  a)u_0)(\cdot)=u_i(t, \cdot;s,  u_0,  a), \quad
u_0\in X_i,
\]
where $u_1(t, \cdot;s, u_0, a)$ is the solution to
\begin{equation}
\label{l-non-kpp-diri}
\partial_t u(t, x)=\nu_\delta\left[ \int_D k_{\delta}(y-x)u(t, y)dy-u(t, x)\right]+a(t, x)u(t, x), \quad x\in \bar D
\end{equation}
with  $u_1(s, \cdot; s, u_0,  a)=u_0(\cdot)\in X_1$, $u_2(t, \cdot;s, u_0, a)$
is the solution to
\begin{equation}
\label{l-non-kpp-neum}
\partial_t u(t, x)=\nu_\delta  \int_D k_{\delta}(y-x)[u(t, y)-u(t, x)]dy+a(t, x)u(t, x),\quad x\in \bar D
\end{equation}
with  $u_2(s, \cdot; s, u_0,  a)=u_0(\cdot)\in X_2$, and $u_3(t, \cdot;s, u_0,
a)$ is the solution to
\begin{equation}
\label{l-non-kpp-peri}
\partial_t u(t, x)=\nu_\delta\left[ \int_{\R^N}k_{\delta}(y-x)u(t, y)dy-u(t, x)\right]+a(t, x)u(t, x), \quad x\in \R^N
\end{equation}
with  $u_3(s, \cdot; s, u_0,  a)=u_0(\cdot)\in X_3$. By general semigroup
property,
$\Phi_i^\delta(t,s;a)$ ($i=1,2,3$) is well defined.

Let $A_1$ be  $-\Delta$ with Dirichlet boundary condition  acting on $X_1\cap C_0(D)$. Let
\begin{equation}
\label{x-1-r-space}
X_1^r=\mathcal{D}(A_1^\alpha)
\end{equation}
for some $0<\alpha<1$ such that $C^1(\bar D)\subset X_1^r$ with $\|u\|_{X_1^r}=\|A_1^{\alpha} u\|_{X_1}$. Similarly, let $A_2$ be $-\Delta$ with Neumann boundary condition acting on $X_2$. Let
\begin{equation}
\label{x-2-r-space}
X_2^r=X_2
\end{equation}
with $\|u\|_{X_2^r}=\|u\|_{X_2}$, and
\begin{equation}
\label{x-3-r-space}
X_3^r=X_3
\end{equation}
with $\|u\|_{X_3^r}=\|u\|_{X_3}$. Let
\[
X_i^{r, +}=\{u\in X_i^r|u(x)\geq 0\}
\]
(i=1, 2, 3). Similarly, for $i=1, 2, 3$, define $ \Phi_i(t, s;  a): X_i^r\to X_i^r$ by
\[
(\Phi_i(t, s;  a)u_0)(\cdot)=u_i(t, \cdot;s, u_0,  a), \quad u_0\in X_i^r,
\]
where $u_1(t, \cdot; s, u_0,  a)$ is the solution to
\begin{equation}
\label{D-r-PE}
\begin{cases}
\p_t u(t, x)=\Delta u(t, x)+a(t,x)u(t, x), \quad & x\in D,\\
u(t,x)=0, & x\in \partial D
\end{cases}
\end{equation}
with  $u_1(s, \cdot; s, u_0,  a)=u_0(\cdot)\in X_1^r$, $u_2(t, \cdot;s, u_0,a)$
is the solution to
\begin{equation}
\label{N-r-PE}
\begin{cases}
\p_t u(t, x)=\Delta u(t, x)+a(t,x)u(t, x), \quad & x\in D,\\
\frac{\partial u}{\partial {\bf n}}(t, x)=0, & x\in \partial D
\end{cases}
\end{equation}
with $u_2(s, \cdot;s, u_0,  a)=u_0(\cdot)\in X_2^r$, and $u_3(t, \cdot; s, u_0,
 a)$ is the  solution to
\begin{equation}
\label{P-r-PE}
\begin{cases}
\p_t u(t, x)=\Delta u(t, x)+a(t,x)u(t, x), \quad & x\in \RR^N,\\
 u(t, x+p_j\bold{e_j})=u(t,x),& x\in \RR^N
\end{cases}
\end{equation}
with $u_3(s, \cdot; s, u_3, a)=u_0(\cdot)\in X_3^r$.

Let  $r(\Phi_i^{\delta}(T,0 ;  a))$ be the spectral radius of
$\Phi_i^{\delta}(T, 0;  a)$ and $\lambda_i^\delta(a)$ be the principal spectrum point of $L_i^\delta(a)$. We have the following propositions.
\begin{proposition}
\label{spectral-radius-nonlocal-prop} Let $1\le i\le 3$ be given. Then
$$r(\Phi_i^{\delta}(T,0;a))=e^{\lambda_i^\delta(a) T}.
$$
\end{proposition}
\begin{proof}
See \cite[Proposition 3.3]{ShXi}.
\end{proof}

Similarly, let $r(\Phi_i(T, 0;  a))$ be the spectral radius of $\Phi_i(T, 0;
 a)$ and $\lambda_i^r(a)$ be the principal eigenvalue of $L_i(a)$. Note that
 $X_i^r$ is a strongly ordered Banach space with the positive cone $C=\{u\in X_i^r\,|\, u(x)\ge 0\}$ and
 by the regularity, {\it a priori} estimates, and comparison principle  for parabolic equations,
 $\Phi_i(T,0;a):X_i^r\to X_i^r$ is strongly positive and compact. Then by
  the Kre\u \i n-Rutman Theorem
(see \cite{K-R}), $r(\Phi_i(T,0;a))$ is an isolated algebraically simple eigenvalue of $\Phi_i(T,0;a)$ with a positive
eigenfunction $u_i^r(\cdot)$ and for any $\mu\in\sigma(\Phi_i(T,0;a))\setminus\{r(\Phi_i(T,0;a))\}$,
$$
\text{Re}\mu< r(\Phi_i(T,0;a)).
$$
The following propositions then follow.

 \begin{proposition}
 \label{spectral-radius-random-prop}
 Let $1\le i\le 3$ be given. Then
 $$
 r(\Phi_i(T,0;a))=e^{\lambda_i^r(a) T}.
 $$
 Moreover, there is a codimension one subspace $Z_i$ of $X_i^r$ such that
 $$
 X_i^r=Y_i\oplus Z_i,
 $$
 where $Y_i={\rm span}\{u_i^r(\cdot)\}$, and there are $M>0$ and $\gamma>0$ such that
 for any $w_i\in Z_i$, there holds
 $$
 \frac{\|\Phi_i(nT,0;a)w_i\|_{X_i^r}}{\|\Phi_i(nT,0;a)u_i^r\|_{X_i^r}}\le M e^{-\gamma nT}.
 $$
 \end{proposition}

 \begin{proposition}
 \label{perturbation-prop}
 For given $1\le i\le 3$ and $a_1,a_2\in \mathcal X_i\cap C^1(\RR\times\RR^N)$,
 \begin{equation}
 \label{perturbation-eq1}
 |\lambda_i^\delta(a_1)-\lambda_i^\delta(a_2)|\le \max_{t\in [0,T], x\in\bar D}|a_1(t,x)-a_2(t,x)|,
 \end{equation}
 and
 \begin{equation}
 \label{perturbation-eq2}
 |\lambda_i^r(a_1)-\lambda_i^r(a_2)|\le \max_{t\in [0,T], x\in\bar D}|a_1(t,x)-a_2(t,x)|.
 \end{equation}
 \end{proposition}

\begin{proof}
Let  $a_0= \max_{t\in [0,T], x\in\bar D}|a_1(t,x)-a_2(t,x)|$ and
$$a_1^\pm(t,x)=a_1(t,x)\pm a_0.$$
It is not difficult to see that
$$
\Phi_i^\delta(t,s;a_1^\pm)=e^{\pm a_0 (t-s)}\Phi_i^\delta(t,s;a_1).
$$
It then follows that
\begin{equation}
\label{perturbation-eq3}
r(\Phi_i^\delta(T,0;a_1^\pm))=e^{(\lambda_i^\delta(a_1)\pm a_0)T}.
\end{equation}
Observe that by Proposition \ref{comparison-prop},  for any $u_0\in X_i^{+}$,
$$
\Phi_i^\delta(T,0;a_1^-)u_0\le \Phi_i^\delta(T,0;a_2)u_0\le \Phi_i^\delta(T,0;a_1^+)u_0.
$$
This implies that
$$
r(\Phi_i^\delta(T,0;a_1^-))\le r(\Phi_i^\delta(T,0;a_2))\le r(\Phi_i^\delta(T,0;a_1^+)).
$$
This together with \eqref{perturbation-eq3} implies that
\begin{equation}
\label{perturbation-eq4}
\lambda_i^\delta(a_1)-a_0\le \lambda_i^\delta(a_2)\le\lambda_i^\delta(a_1)+a_0,
\end{equation}
that is, \eqref{perturbation-eq1} holds.

Similarly, we can prove that \eqref{perturbation-eq2} holds.
\end{proof}

\subsection{Proof of Theorem B in the Dirichlet boundary condition case}

In this subsection, we prove Theorem B in the Dirichlet  boundary
condition case. Throughout this subsection, we assume
$B_{r,b}u=B_{r,D}u$ in \eqref{main-random-eigenvalue}, and
$D_b=D_D(=\RR^N\setminus \bar D)$ and $B_{n,b}u=B_{n,D}u$ in
\eqref{main-nonlocal-eigenvalue}.
Note that for any $a\in \mathcal{X}_1\cap C^1(\RR\times\RR^N)$, there are $a_n\in \mathcal{X}_1\cap C^3(\RR\times\RR^N)$ such that
$\sup_{t\in [0,T]}\|a_n(t,\cdot)-a(t,\cdot)\|_{X_1} \to 0$ as $n\to\infty$. By Proposition \ref{perturbation-prop}, without loss
of generality, we may assume that $a\in\mathcal{X}_1\cap C^3(\RR\times\RR^N)$.

\begin{proof} [Proof of Theorem B in the Dirichlet  boundary
condition case]

First of all, for the simplicity in notation, we put
$$
\Phi(T,0)=\Phi_1(T,0;a), \,\ \quad \lambda_1^r=\lambda_1^r(a),
$$
and
$$
\Phi^\delta(T,0)=\Phi_1^\delta(T,0;a), \quad \lambda_1^\delta=\lambda_1^\delta(a).
$$
Let $u^r(\cdot)$ be a positive eigenfunction of $\Phi(T,0)$ corresponding to
$r(\Phi(T,0))$. Without loss of generality, we assume that $\|u^r\|_{X_1^r}=1$.

We first show that for any $\epsilon>0$, there is $\delta_1>0$ such that for $0<\delta<\delta_1$,
\begin{equation}
\label{thmb-d-eq1}
\lambda_1^\delta\geq \lambda_1^r-\epsilon.
\end{equation}
In order to do so, choose $D_0\subset \subset D$ and $u_0\in  X_1^r\cap C^3(\bar D)$ such that $u_0(x)=0$ for $x\in D\backslash D_0$, and $u_0(x)>0$ for $x\in \text{Int}D_0$.
 By Proposition \ref{spectral-radius-random-prop},  there exist $\alpha> 0$, $M>0$, and $u'\in Z_1$, such that
\begin{equation}
\label{thmb-d-eq2}
u_0(x)=\alpha u^r(x)+u'(x),
\end{equation}
and
\begin{equation}
\label{thmb-d-eq3}
\|\Phi(nT,0)u'\|_{X_1^r}\le M e^{-\gamma nT} e^{\lambda_1^r nT}.
\end{equation}
By Theorem A, there is $\delta_0>0$ such that for $0<\delta<\delta_0$, there hold
\begin{equation}
\label{thmb-d-eq4}
\big(\Phi^{\delta}(nT, 0)u^r\big)(x)\geq \big(\Phi(nT, 0)u^r\big)(x)-C^1(nT,\delta),
\end{equation}
and
\begin{equation}
\label{thmb-d-eq5}
\big(\Phi^{\delta}(nT, 0)u'\big)(x)\leq \big(\Phi(nT, 0)u'\big)(x)+C^2(nT,\delta),
\end{equation}
where $C^i(nT,\delta)\to 0$ as $\delta\to 0$ ($i=1,2$).
Hence for $0<\delta<\delta_0$,
\begin{align}
\label{thmb-d-eq6}
\big(\Phi^\delta(nT,0)u_0\big)(x)=&\alpha\big(\Phi^{\delta}( n T, 0) u^r\big)(x)+\big(\Phi^{\delta}( n T, 0)u'\big)(x)\notag\\
\geq & \alpha \big(\Phi(nT, 0)u^r\big)(x)-\alpha C^1(nT,\delta)-C^2(nT, \delta)-\|\Phi(nT, 0)u'\|_{X_1^r}\notag\\
\geq & \alpha  e^{\lambda_1^r nT}u^r(x)-\alpha C^1(nT,\delta)-C^2(nT,\delta)-M e^{-\gamma nT} e^{\lambda_1^r nT}\notag\\
= &  e^{(\lambda_1^r-\epsilon) nT} e^{\epsilon nT}(\alpha u^r(x)-Me^{-\gamma nT})-\alpha C^1(nT,\delta)-C^2(nT,\delta).
\end{align}
Note that  there exists $m>0$ such that
 \begin{equation*}
u^r(x)\geq m>0\quad \text{ for } x\in \bar D_0.
\end{equation*}
Hence for any $0<\epsilon<\gamma$, there is $n_1>0$ such that for $n\ge n_1$,
\begin{equation}
\label{thmb-d-eq7}
e^{\epsilon nT} (\alpha u^r (x)-M e^{-\gamma nT})\ge u_0(x)+1\quad {\rm for}\quad x\in \bar D_0,
\end{equation}
and there is $\delta_1\le \delta_0$ such that for $0<\delta<\delta_1$,
\begin{equation}
\label{thmb-d-eq8}
C^1(n_1T,\delta)+C^2(n_1T,\delta)\le  e^{(\lambda_1^r-\epsilon)n_1T}.
\end{equation}
Note that $u_0(x)=0$ for $x\in D\backslash D_0$ and $\big(\Phi^\delta(n_1T,0)u_0\big)(x)\ge 0$ for all $x\in \bar D$. This together with
 \eqref{thmb-d-eq6}-\eqref{thmb-d-eq8} implies that for $\delta<\delta_1$,
\begin{equation}
\label{thmb-d-eq9}
\big(\Phi^\delta(n_1T,0)u_0\big)(x)\ge e^{(\lambda_1^r-\epsilon)n_1T} u_0(x),\quad x\in\bar D.
\end{equation}
By \eqref{thmb-d-eq9} and Proposition \ref{comparison-prop}, for any  $0<\delta<\delta_1$ and $n\ge 1$,
$$
(\Phi^\delta(nn_1T,0)u_0)(\cdot)\ge e^{(\lambda_1^r-\epsilon)nn_1T} u_0(\cdot).
$$
This together with Proposition \ref{spectral-radius-nonlocal-prop} implies that for $0<\delta<\delta_1$,
$$
e^{\lambda_1^\delta T}=r(\Phi^\delta(T,0))\ge e^{(\lambda_1^r-\epsilon)T}.
$$
Hence \eqref{thmb-d-eq1} holds.

Next, we prove that for any $\epsilon>0$, there is $\delta_2>0$ such that for $0<\delta<\delta_2$,
\begin{equation}
\label{thmb-d-eq10}
\lambda_1^\delta\leq \lambda_1^r+ \epsilon.
\end{equation}
To this end, first,
choose  a sequence of smooth
domains $\{D_m\}$ such that $ D_1\supset  D_2 \supset  D_3 \cdots \supset  D_m \supset \cdots \supset \bar D$,
  and $\cap_{m=1}^{\infty}D_m=\bar D$. Consider the following evolution equation
\begin{equation}
\label{D-r-PE-hat}
\begin{cases}
\p_tu(t, x)= \Delta  u(t, x)+a(t,x) u(t, x), \quad  & x\in D_m,\cr
 u(t,x)=0, \quad  & x \in \partial  D_m.
\end{cases}
\end{equation}
Let
\[
X_{1, m}=\{u\in C(\bar D_m, \R)\},
\]
and
\[
X_{1, m}^r=\mathcal D(A_{1, m}^\alpha),
\]
where $A_{1, m}$ is  $-\Delta$ with Dirichlet boundary condition acting on $X_{1, m}\cap C_0(D_m)$ and $0<\alpha<1$.
We denote the solution of \eqref{D-r-PE-hat} by $ u_m(t, \cdot; s,  u_0)=(\Phi_m(t, s)  u_0)(\cdot)$ with $u(s, \cdot; s,  u_0)= u_0(\cdot)\in X_{1, m}^r$.
By Proposition \ref{spectral-radius-random-prop}, we have
\begin{equation*}
r(\Phi_m(T, 0))=e^{\lambda_{1, m}^rT},
\end{equation*}
where $\lambda^r_{1, m}$ is the principal eigenvalue of  the following eigenvalue problem,
\begin{equation*}
\begin{cases}
-\p_t u+\Delta u+a(t,x)u=\lambda u,\quad & x\in D_m,\cr
u(t+T,x)=u(t,x),\quad & x\in D_m,\cr
u(t,x)=0,\quad & x\in\p D_m.
\end{cases}
\end{equation*}
By the dependence of the principle eigenvalue on the domain perturbation (see \cite{Dan}), for any $\epsilon>0$, there exists $m_1$ such that
\begin{equation}
\label{thmb-d-eq11}
\lambda_{1, m_1}^r\leq \lambda_1^r+\frac{\epsilon}2.
\end{equation}

Secondly,
let $u_{m_1}^r(\cdot)$ be a positive eigenfunction of $\Phi_{m_1}(T,0)$ corresponding to $r(\Phi_{m_1}(T,0))$. By regularity for parabolic equations, $u_{m_1}^r\in C^3(\bar D_{m_1})$.
Let $ (\Phi_{m_1}^\delta(t, 0) u^r_{m_1})(x)$ be the solution to
\begin{equation}
\label{D-non-hat}
\begin{cases}
u_t=\nu_\delta\left[\int_{ D_{m_1}} k_\delta(y-x) u(t, y)dy- u(t,x)\right]+a(t,x)u(t,x), \quad& x\in  \bar D_{m_1}, \\
u(0,x)=u_{m_1}^r(x).
\end{cases}
\end{equation}
Then by Theorem A,
$$\big(\Phi_{m_1}^\delta(nT,0) u_{m_1}^r\big)(x)\leq\big( \Phi_{m_1}(nT,0) u_{m_1}^r\big)(x)+C(nT,\delta)\quad \forall \,\, x\in\bar D_{m_1},
$$
where $C(nT,\delta)\to 0$ as $\delta\to 0$.
By Proposition \ref{comparison-prop},
$$
\big(\Phi^\delta(nT,0)u_{m_1}^r|_{\bar D}\big)(x)\le \big(\Phi_{m_1}^\delta(nT,0)u_{m_1}^r\big)(x)\quad \forall \,\, x\in\bar D.
$$
It then follows that for $x\in\bar D$,
\begin{align}
\label{thmb-d-eq12}
\big(\Phi^\delta(nT,0)u_{m_1}^r|_{\bar D}\big)(x)&\le \big(\Phi_{m_1}(nT,0) u_{m_1}^r\big)(x)+C(nT,\delta)\nonumber\\
&=e^{\lambda_{m_1}^r nT} u_{m_1}^r(x)+C(nT,\delta)\nonumber\\
&\le e^{(\lambda_1^r +\frac{\epsilon}{2})nT}u_{m_1}^r(x)+C(nT,\delta)\nonumber\\
&= e^{(\lambda_1^r+\epsilon)nT} e^{-\frac{\epsilon}{2}nT}u_{m_1}^r(x)+C(nT,\delta).
\end{align}
Note that
$$
\min_{x\in\bar D} u_{m_1}^r(x)>0.
$$
Hence for any $\epsilon>0$, there is $n_2\ge 1$ such that
\begin{equation}
\label{thmb-d-eq13}
e^{-\frac{\epsilon }{2}n_2T}\le \frac{1}{2},
\end{equation}
and there is $\delta_2>0$ such that for $0<\delta<\delta_2$,
\begin{equation}
\label{thmb-d-eq14}
C(n_2T,\delta)\le \frac{1}{2} e^{(\lambda_1^r+\epsilon)n_2T} u_{m_1}^r(x) \quad \forall \, x\in \bar D.
\end{equation}
By \eqref{thmb-d-eq12}-\eqref{thmb-d-eq14},
$$
\big(\Phi^\delta(n_2T,0)u_{m_1}^r|_{\bar D}\big)(x)\le e^{(\lambda_1^r+\epsilon)n_2T}u_{m_1}^r(x)\quad \forall \,\, x\in\bar D.
$$
This together with Proposition \ref{comparison-prop} implies that for $0<\delta<\delta_2$,
\begin{equation}
\label{thmb-d-eq15}
\big(\Phi^\delta(nn_2T,0)u_{m_1}^r|_{\bar D}\big)(x)\le e^{(\lambda_1^r+\epsilon)nn_2T} u_{m_1}^r(x)\quad \forall \,\, x\in\bar D.
\end{equation}
This together with Proposition \ref{spectral-radius-nonlocal-prop} implies that
$$
\lambda_1^\delta \le\lambda_1^r+\epsilon
$$
for $0<\delta<\delta_2$, that is,
\eqref{thmb-d-eq10} holds.

Theorem B in the Dirichlet boundary condition case then follows from \eqref{thmb-d-eq1} and \eqref{thmb-d-eq10}.
\end{proof}

\subsection{Proofs of Theorem B in the Neumann and periodic  boundary condition
cases}

\begin{proof} [Proof of Theorem B in the Neumann  boundary condition  case]
We assume
$B_{r,b}u=B_{r,N}u$ in \eqref{main-random-eigenvalue}, and
$D_b=D_N(=\emptyset)$ in
\eqref{main-nonlocal-eigenvalue}.
The proof  in the Neumann boundary condition case is similar to the arguments in the Dirichlet boundary condition case 
 (it is simpler). For the completeness, we give a proof in the
following. Without loss of generality, we may also assume that $a\in \mathcal{X}_2\cap C^3(\RR\times\RR^N)$.

For the simplicity in notation, put
$$
\Phi(nT,0)=\Phi_2(nT,0; a), \quad \,\ \lambda_2^r=\lambda_2^r(a),
$$
and
$$
\Phi^\delta(nT,0)=\Phi_2^\delta(nT,0;a), \quad  \lambda_2^\delta=\lambda_2^\delta(a).
$$
By Propositions \ref{spectral-radius-nonlocal-prop} and \ref{spectral-radius-random-prop},
\begin{equation}
\label{thmb-n-eq2}
r(\Phi(T,0))=e^{\lambda_2^r T},
\end{equation}
and
\begin{equation}
\label{thmb-n-eq1}
r(\Phi^\delta(T,0))=e^{\lambda_2^\delta T}.
\end{equation}

Let  $u^r(\cdot)$ be a positive eigenfunction of $\Phi(T,0)$  corresponding
to $r(\Phi(T,0))$. By regularity for parabolic equations, $u^r\in C^3(\bar D)$.
  By Theorem A, we have
\begin{equation*}
\|\Phi^\delta(nT, 0)u^r-\Phi(nT, 0)u^r\|_{X_2}\leq C(nT,\delta),
\end{equation*}
where $C(nT,\delta)\to 0$ as $\delta\to 0$.
This implies that for all $x\in\bar D$,
\begin{align}
\label{thmb-n-eq3}
\big(\Phi^\delta(nT, 0)u^r\big)(x)&\geq \big(\Phi(nT, 0)u^r\big)(x)-C(nT,\delta)\nonumber\\
&=e^{\lambda_2^rnT}u^r(x)-C(nT,\delta)\nonumber\\
&=e^{(\lambda_2^r-\epsilon)nT} e^{\epsilon nT} u^r(x)-C(nT,\delta),
\end{align}
and
\begin{align}
\label{thmb-n-eq4}
\big(\Phi^\delta(nT, 0)u^r\big)(x)&\leq \big(\Phi(nT, 0)u^r\big)(x)+C(nT,\delta)\nonumber\\
&=e^{\lambda_2^r nT}u^r(x)+C(nT,\delta)\nonumber\\
&=e^{(\lambda_2^r+\epsilon)nT} e^{-\epsilon nT}u^r(x)+C(nT,\delta).
\end{align}
Note that
\begin{equation}
\label{thmb-n-eq5}
\min_{x\in\bar D}u^r(x)>0.
\end{equation}
Hence for any $\epsilon>0$, there is $n_1>1$ such that
\begin{equation}
\label{thmb-n-eq6}
\begin{cases}
e^{\epsilon n_1 T}u^r(x)\ge \frac{3}{2} u^r(x)\quad \forall \, x\in\bar D,\\
\\
e^{-\epsilon n_1 T} u^r(x)\le \frac{1}{2} u^r(x)\quad \forall\, x\in\bar D,
\end{cases}
\end{equation}
and  there is $\delta_0>0$ such that for any $0<\delta<\delta_0$,
\begin{equation}
\label{thmb-n-eq7}
C(n_1T)\delta<\frac{1}{2} e^{(\lambda_2^r -\epsilon)n_1T} u^r(x)\quad \forall\, x\in\bar D.
\end{equation}
By \eqref{thmb-n-eq3}-\eqref{thmb-n-eq7},
we have that for any $0<\delta<\delta_0$,
\begin{equation*}
e^{(\lambda_2^r -\epsilon)n_1T}u^r(x)\le\big(\Phi^\delta(n_1T,0)u^r\big)(x)\le e^{(\lambda_2^r+\epsilon)n_1T}u^r(x)\quad \forall\, x\in\bar D.
\end{equation*}
This together with Proposition \ref{comparison-prop} implies that for all $n\ge 1$,
$$
e^{(\lambda_2^r -\epsilon)n_1nT}u^r(x)\le\big(\Phi^\delta(n_1 nT,0)u^r\big)(x)\le e^{(\lambda_2^r+\epsilon)n_1nT}u^r(x)\quad \forall\, x\in\bar D.
$$
It then follows that for any $0<\delta<\delta_0$,
$$
e^{(\lambda_2^r-\epsilon)T}\le r(\Phi^\delta(T,0))\le e^{(\lambda_2^r+\epsilon)T}.
$$
By Proposition \ref{spectral-radius-nonlocal-prop}, we have
$$
|\lambda_2^\delta-\lambda_2^r|<\epsilon\quad \forall\, 0<\delta<\delta_0.
$$
Theorem B in the Neumann boundary condition case is thus proved.
\end{proof}

\begin{proof} [Proof of Theorem B in the periodic boundary condition case] We assume $D=\R^N$, and
$B_{r,b}u=B_{r,P}u$ in \eqref{main-random-eigenvalue}, and $B_{n,b}u=B_{n,P}u$ in
\eqref{main-nonlocal-eigenvalue}.
It can be proved by the same arguments as in the Neumann boundary condition case.
\end{proof}

\section{Approximation of Time Periodic Positive Solutions of Random Dispersal KPP Equations
by Nonlocal Dispersal KPP Equations}

In this section, we study the approximation of the asymptotic dynamics of time periodic KPP equations with random dispersal
by those of time periodic KPP equations with nonlocal dispersal. We first recall the existing results about time periodic
positive solutions of KPP equations with random as well as nonlocal dispersal. Then we prove Theorem C.
Throughout this section, we assume that  $D$ is as in (H0), and (H1), (H2) and (H2)$_\delta$ hold. Recall that, (H2) implies (H2)$_\delta$ for $\delta$ sufficiently small by Theorem B.

\subsection{Basic properties}

In this subsection, we present some basic known results for \eqref{main-random-kpp} and
\eqref{main-nonlocal-kpp}.
Let $X_1^r$, $X_2^r$, and $X_3^r$ be defined as in \eqref{x-1-r-space}, \eqref{x-2-r-space}, and \eqref{x-3-r-space},
respectively. For $u_0\in X_i^r$, let $(U(t,0)u_0)(\cdot)=u(t,\cdot;u_0)$, where $u(t,\cdot;u_0)$ is the solution
to \eqref{main-random-kpp} with $u(0,\cdot;u_0)=u_0(\cdot)$ and $B_{r,b}u=B_{r,D}u$  when $i=1$,
$B_{r,b}u=B_{r,N}u$ when $i=2$, and $B_{r,b}u=B_{r,P}u$  when $i=3$.
Similarly,  for $u_0\in X_i$, let $(U^\delta(t,0)u_0)(\cdot)=u^\delta(t,\cdot;u_0)$, where $u^\delta(t,\cdot;u_0)$ is the solution
to \eqref{main-nonlocal-kpp} with $u^\delta(0,\cdot;u_0)=u_0(\cdot)$ and $D_b=D_D(=\RR^N\backslash \bar D)$, $B_{n,b}u=B_{n,D}u$ when $i=1$,
$D_b=D_N(=\emptyset)$ when $i=2$, and  $B_{n,b}u=B_{n,P}u$ and $D_b=D_p(=\RR^N)$ when $i=3$.

\begin{proposition}
\label{positive-global-existence}
\begin{itemize}
\item[(1)] If $u_0\ge 0$, solution $u(t,\cdot;u_0)$ to
\eqref{main-random-kpp} with $u(0,\cdot;u_0)=u_0(\cdot)$ exists for all
$t\ge 0$ and $u(t,\cdot;u_0)\ge 0$ for all $t\ge 0$.

\item[(2)]  If $u_0\ge 0$, solution $u(t,\cdot;u_0)$ to
\eqref{main-nonlocal-kpp} with $u(0,\cdot;u_0)=u_0(\cdot)$ exists for all
$t\ge 0$ and $u(t,\cdot;u_0)\ge 0$ for all $t\ge 0$.
\end{itemize}
\end{proposition}

\begin{proof}
(1) Note that $u(\cdot)\equiv 0$ is a sub-solution of \eqref{main-random-kpp} and $u(\cdot)\equiv M$ is a super-solution of \eqref{main-random-kpp} for
$M\gg 1$. Then by Proposition \ref{comparison-prop}, there is $M\gg 1$ such that
$$
0\le u(t,x;u_0)\le M\quad \forall\,\, x\in\bar D,\,\, t\in (0,t_{\max}),
$$
where $(0,t_{\max})$ is the  interval of existence of $u(t,\cdot;u_0)$. This implies that we must have $t_{\max}=\infty$ and hence
(1) holds.

(2) It can be proved by similar arguments as in (1).
\end{proof}

\begin{proposition}
\label{time-periodic-solu}
\begin{itemize}
\item[(1)] \eqref{main-random-kpp} has a unique globally stable positive  time
periodic  solution $u^*(t,x)$.

\item[(2)] \eqref{main-nonlocal-kpp} has a unique globally stable
time periodic positive solution $u^*_\delta(t,x)$.
\end{itemize}
\end{proposition}

\begin{proof}
(1) See   \cite[Theorem 3.1]{Zhaox} (see also \cite[Theorems 1.1, 1.3]{Nad1}).

(2) See \cite[Theorem E]{RaSh}.
\end{proof}

\begin{remark}
\label{time-periodic-solu-rk}
By Proposition \ref{time-periodic-solu}(2), if there is $u_0\in X_i^+\setminus\{0\}$ such that
$(U^\delta(nT,0)u_0)(\cdot)\ge u_0(\cdot)$ for some $n\ge 1$, then we must have $\lim_{n\to\infty}(U^\delta(nT,0)u_0)(\cdot)=u^*_\delta(0,\cdot)$ and
hence
$$
(U^\delta(nT,0)u_0)(\cdot)\le u^*_\delta(0,\cdot).
$$
Similarly, if there is $u_0\in X_i^+\setminus \{0\}$ such that $(U^\delta(nT,0)u_0)(\cdot)\le u_0(\cdot)$ for some $n\ge 1$, then
$$
(U^\delta(nT,0)u_0)(\cdot)\ge u^*_\delta(0,\cdot).
$$
\end{remark}

\subsection{Proof of Theorem C in the Dirichlet boundary condition
case}

In this subsection, we prove Theorem C in the Dirichlet boundary condition case. Throughout this subsection,
we assume that $B_{r,b}u=B_{r,D}u$ in \eqref{main-random-kpp}, and $D_b=D_D$ and
$B_{n,b}u=B_{n,D}u$ in \eqref{main-nonlocal-kpp}.

 \begin{proof} [Proof  of Theorem C in the Dirichlet boundary condition
case]

First of all, note that it suffices to prove that for any $\epsilon
>0$, there is $\delta_0>0$ such that for $0<\delta<\delta_0$,
\[
 u_{\delta}^*(t, x)-\epsilon\le u^*(t, x)\leq u_{\delta}^*(t,
 x)+\epsilon\quad \forall \,\, t\in [0,T],\,\, x\in\bar D.
\]

We first show that  for any $\epsilon>0$, there is $\delta_1>0$ such that for $0<\delta<\delta_1$,
\begin{equation}
\label{thmc-d-eq0} u^*(t, x)\leq u_{\delta}^*(t, x)+\epsilon\quad \forall\,\, t\in [0,T],\,\, x\in\bar D.
\end{equation}
To this end,  choose a smooth  function $\phi_0\in C_0^\infty(D)$ satisfying that $\phi_0(x)\ge 0$ for $x\in D$ and
$\phi_0(\cdot)\not \equiv 0$. Let $0<\eta\ll 1$ be such that
\begin{equation*}
u_-(x):=\eta \phi_0(x)< u^*(0,x)\quad {\rm for}\quad x\in\bar D.
\end{equation*}
Then there is $\epsilon_0>0$ such that
\begin{equation}
\label{thmc-d-eq1-1}
u^*(0,x)\ge u_-(x)+\epsilon_0\quad {\rm for}\quad x\in {\rm supp}(\phi_0).
\end{equation}
By Proposition \ref{time-periodic-solu}, there is $N\gg 1$ such that
$$
\big(U(NT,0)u_-\big)(x)\ge u^*(NT,x)-\epsilon_0/2=u^*(0,x)-\epsilon_0/2\quad \forall\,\, x\in\bar D.
$$
By Theorem A, there is $\bar\delta_1>0$ such that for $0<\delta<\bar\delta_1$, we have
$$
\big(U^\delta(NT,0)u_-\big)(x)\ge \big(U(NT,0)u_-\big)(x)-\epsilon_0/2\quad \forall\,\, x\in\bar D.
$$
Hence for $0<\delta<\bar\delta_1$,
\begin{equation}
\label{thmc-d-eq1-2}
\big(U^\delta(NT,0)u_-\big)(x)\ge u^*(0,x)-\epsilon_0\quad \forall \,\, x\in\bar D.
\end{equation}
Note that
$$
\big(U^\delta(NT,0)u_-\big)(x)\ge 0\quad \forall\,\, x\in\bar D.
$$
It then follows  from \eqref{thmc-d-eq1-1} and \eqref{thmc-d-eq1-2} that for $0<\delta<\bar \delta_1$,
\begin{equation*}
\big(U^\delta(NT,0)u_-\big)(x)\ge u_-(x)\quad \forall \,\, x\in\bar D.
\end{equation*}
 This together with Proposition \ref{time-periodic-solu} (2) implies that
\begin{equation}
\label{thmc-d-eq1}
\big(U^\delta(NT,0)u_-\big)(x)\le u_\delta^*(0,x)\quad \forall\,\, x\in\bar D
\end{equation}
(see Remark \ref{time-periodic-solu-rk}).

 By Proposition \ref{time-periodic-solu} (1) again, for $n\gg 1$,
 \begin{equation}
 \label{thmc-d-eq2}
u^*(t, x) \leq (U(nNT+t, 0)u_-)(x)+\epsilon/2\,\,\,\, \forall \,\, t\in [0,T],\,\, x\in\bar D.
\end{equation}
Fix an $n\gg 1$ such that \eqref{thmc-d-eq2} holds.
By Theorem A, there is $0<\tilde \delta_1\le\bar \delta_1$ such that  for $0<\delta<\tilde \delta_1$,
\begin{align}
\label{thmc-d-eq3}
 (U(nNT+t, 0)u_-)(x)\leq (U^{\delta}(nNT+t, 0)u_-)(x)+C_1(\delta),
\end{align}
where $C_1(\delta)\to 0$ as $\delta\to 0$.
By \eqref{thmc-d-eq1}, Proposition \ref{comparison-prop},  and  Proposition
\ref{time-periodic-solu} (2),
\begin{equation}
\label{thmc-eq4}
\big(U^{\delta}(nNT+t, 0)u_-\big)(x)\le  \big(U^{\delta}(t, 0)u^*_{\delta}(0,\cdot)\big)(x)=u_\delta^*(t,x)
\end{equation}
for $t\in[0,T]$ and $x\in\bar D$. Let $0<\delta_1\le\tilde\delta_1$ be such that
\begin{equation}
\label{thmc-d-eq5}
C_1(\delta)<\epsilon/2\quad \forall\,\, 0<\delta<\delta_1.
\end{equation}
\eqref{thmc-d-eq0} then follows from \eqref{thmc-d-eq2}-\eqref{thmc-d-eq5}.

Next, we need to show for any $\epsilon>0$, there is $\delta_2>0$ such that for $0<\delta<\delta_2$,
\begin{equation}
\label{thmc-d-eq6} u^*(t, x)\geq u_{\delta}^*(t, x)-\epsilon\quad \forall\,\, t\in [0,T],\,\, x\in\bar D.
\end{equation}
To this end, choose a sequence of open sets $\{D_m\}$ with smooth
boundaries such that $D_1\supset  D_2 \supset  D_3 \cdots \supset  D_m \supset \cdots \supset \bar D$, and
$\bar D=\cap_{m=1}^\infty D_m$. According to
Corollary 5.11 in \cite{DanersA},  $D_m\to D$ regularly and
all assertions of Theorem 5.5 in \cite{DanersA} hold.

Consider
\begin{equation}
\label{thmc-d-eq7}
\begin{cases}
\partial _t u= \Delta u + uf(t, x, u), \quad & x\in  D_m,\\
u(t, x)=0, \quad & x \in \partial  D_m.
\end{cases}
\end{equation}
Let $U_m(t,0)u_0=u(t,\cdot;u_0)$, where $u(t,\cdot;u_0)$ is the solution to \eqref{thmc-d-eq7} with $u(0,\cdot;u_0)=u_0(\cdot)$.
By Proposition \ref{time-periodic-solu}, \eqref{thmc-d-eq7} has a unique time periodic positive solution $u_m^*(t,x)$.
We first claim that
\begin{equation}
\label{thmc-d-eq8}
\lim_{m\to\infty} u_m^*(t,x)\to u^*(t,x)\,\,\, \text{uniformly in}\,\, t\in [0,T]\,\, \text{and}\,\,  x\in\bar D.
\end{equation}

In fact, it is clear that
 $u^*\in C(\RR\times\bar D,\RR)$
and
$u_m^*\in C(\RR\times\bar D_m,\RR)$. By \cite[Theorem 7.1]{Dan},
$$
\sup_{t\in\RR}\|u_m^*(t,\cdot)-u^*(t,\cdot)\|_{L^q(D)}\to 0\quad {\rm as}\quad m\to\infty
$$
for $1\le q<\infty$.
Let $a(t,x)=f(t,x,u^*(t,x))$ and $a_m(t,x)=f(t,x,u_m^*(t,x))$. Then $u^*(t,x)$ and $u_m^*(t,x)$ are time periodic solutions
to the following linear parabolic equations,
\begin{equation}
\label{linear-eq-on-omega}
\begin{cases}
u_t=\Delta u+a(t,x)u,\quad& x\in D,\cr
u(t,x)=0,& x\in\p D,
\end{cases}
\end{equation}
and
\begin{equation}
\label{linear-eq-on-omega-n}
\begin{cases}
u_t=\Delta u+a_m(t,x)u,\quad& x\in D_m,\cr
u(t,x)=0,& x\in\p D_m,
\end{cases}
\end{equation}
respectively. Observe that there is $M>0$, such that
\[\|a\|_{L^\infty([T, 2T]\times D)}<M, \,\,  \|a_m\|_{L^\infty([T, 2T]\times D_m)}<M,\,\, \|u^*(0,\cdot)\|_{L^\infty(D)}<M,
\text{  and  }
\|u_m^*(0,\cdot)\|_{L^\infty(D_m)}<M.
\]
By \cite[Theorem D(1)]{Aro}, $\{u_m^*(t,x)\}$ is equi-continuous on $[T,2T]\times\bar D$.
Without loss of generality, we may then assume that $u_m^*(t,x)$ converges uniformly on $[T,2T]\times \bar D$.
But $u_m^*(t,\cdot)\to u^*(t,\cdot)$ in $L^q(D)$ uniformly in $t$. We then must have
$$
u_m^*(t,x)\to u^*(t,x)\quad {\rm as}\quad n\to\infty
$$
uniformly in $(t,x)\in [T,2T]\times \bar D$. This together with the time periodicity of $u_m^*$ shows that \eqref{thmc-d-eq8} holds.

Next, for any $\epsilon>0$, fix $m\gg 1$ such that
\begin{equation}
\label{thmc-d-eq9}
u^*(t,x)\ge u_m^*(t,x)-\epsilon/3\quad\forall \,\, t\in [0,T],\,\, x\in\bar D.
\end{equation}
Choose $M\gg 1$ such that for $0<\delta\le 1$,
\begin{equation}
\label{thmc-d-eq10}
M u_m^*(t,x)\ge u_\delta^*(t,x)\,\,\,\, \forall \,\, t\in [0,T],\,\, x\in \bar D.
\end{equation}
Let
$$
u_m^+(x)=M u_m^*(0,x),\quad u^+(x)=u_m^+(x)|_{\bar D}.
$$
By Proposition \ref{time-periodic-solu}, for fixed $m$ and $\epsilon$, there exists $N\gg 1$, such that
\begin{equation}
\label{thmc-d-eq11}
u_m^*(t, x)\geq \big(U_m(NT+t, 0)u_m^+\big)(x)-\epsilon/3\quad \forall\,\, t\in [0,T],\,\, x\in\bar D.
\end{equation}
By Theorem A, there is $0<\tilde\delta_2<1$   such that for $0<\delta<\tilde \delta_2$,
\begin{equation}
\label{thmc-d-eq12}
(U_m(NT+t, 0)u_m^+)(x)\geq (U_m^{\delta}(NT+t, 0)u_m^+)(x)-C_2(\delta) \quad \forall \, t\in [0, T],  \,\ x\in D_m,
\end{equation}
where $C_2(\delta)\to 0$ as $\delta\to 0$ and $(U^\delta_m(t,0)u_0)(\cdot)=u(t,\cdot;u_0)$ is the solution to
$$
\begin{cases}
u_t(t,x)=\nu_\delta\left[\int_{D_m}k_\delta(y-x)u(t,y)dy-u(t,x)\right]+u(t,x)f(t,x,u(t,x)),\quad &x\in\bar D_m\cr
u(0,x)=u_0(x),& x\in\bar D_m.
\end{cases}
$$
Let $0<\delta_2<\tilde \delta_2$ be such that for $0<\delta<\delta_2$,
\begin{equation}
\label{thmc-d-eq13}
C_2(\delta)<\epsilon/3.
\end{equation}
By Proposition \ref{comparison-prop}, for $x\in \bar D$ we have
\[
(U_m^{\delta}(NT+t, 0)u_m^+)(x)\geq (U^{\delta}(NT+t, 0)u^+)(x),
\]
and
\[
 (U^{\delta}(NT+t, 0)u^+)(x)=(U^{\delta}(t, 0)U^{\delta}(NT, 0)u^+)(x)\geq (U^{\delta}(t, 0)u_{\delta}^*(0,\cdot))(x)=u_{\delta}^*(t, x).
\]
This together with \eqref{thmc-d-eq9}, \eqref{thmc-d-eq11}, \eqref{thmc-d-eq12}, and \eqref{thmc-d-eq13} implies
\eqref{thmc-d-eq6}.

So, for any $\epsilon>0$, there exists $\delta_0=\min\{\delta_1, \delta_2\}$, such that for any $\delta<\delta_0$, we have
\[
|u^*(t, x)-u_{\delta}^*(t, x)|\leq \epsilon\,\ \text{ uniform in } t>0 \text{ and } x\in \bar D.
\]
\end{proof}

\subsection{Proofs of Theorem C in the Neumann and periodic boundary
condition cases}

In this subsection, we prove Theorem C in the Neumann and periodic boundary condition cases.

\begin{proof} [Proof of Theorem C in the Neumann boundary condition case] We assume
$B_{r,b}u=B_{r,N}u$ in \eqref{main-random-eigenvalue}, and $D_b=D_N(=\emptyset)$  in
\eqref{main-nonlocal-eigenvalue}.
The proof  in the Neumann boundary condition case is similar to the arguments in the Dirichlet boundary condition case
(it is indeed simpler).
For completeness, we provide a proof.

First,  we  show  that for any $\epsilon>0$, there is $\delta_1>0$ such that
\begin{equation}
\label{thmc-n-eq0}
u^*(t, x)\leq u_{\delta}^*(t, x)+\epsilon\quad \forall\,\, t\in [0,T],\,\, x\in\bar D,
\end{equation}
if $0<\delta<\delta_1$.
Choose a smooth function $u_-\in C^\infty(\bar D)$ with $u_-(\cdot)\ge 0$ and $u_-(\cdot)\not\equiv 0$ such that
$$
u_-(x)<u^*(0,x)\quad \forall \,\, x\in \bar D.
$$
Then there is $\epsilon_0>0$ such that
\begin{equation}
\label{thmc-n-eq1-0}
u^*(0,x)\ge u_-(x)+\epsilon_0\quad \forall \,\, x\in\bar D.
\end{equation}
By Proposition \ref{time-periodic-solu} (1), there is $N\gg 1$ such that
\begin{equation}
\label{thmc-n-eq1-1}
\big(U(NT,0)u_-\big)(x)\ge u^*(0,x)-\epsilon_0/2\quad \forall\,\, x\in\bar D.
\end{equation}
By Theorem A, there is $\bar\delta_1>0$ such that for $0<\delta<\bar\delta_1$,
\begin{equation}
\label{thmc-n-eq1-2}
(U^\delta(NT,0)u_-)(x)\ge (U(NT,0)u_-)(x)-\epsilon_0/2\quad \forall\,\, x\in\bar D.
\end{equation}
By \eqref{thmc-n-eq1-0}, \eqref{thmc-n-eq1-1} and \eqref{thmc-n-eq1-2},
\begin{equation*}
\big(U^\delta(NT,0)u_-\big)(x)\ge u_-(x)\quad \forall\,\, x\in\bar D,
\end{equation*}
and then  by Proposition \ref{time-periodic-solu} (2),
\begin{equation}
\label{thmc-n-eq1}
\big(U^\delta(NT,0)u_-\big)(x)\le u_\delta^*(0,x)\quad \forall\,\, x\in\bar D.
\end{equation}

By
Proposition \ref{time-periodic-solu} (1) again, for any given $\epsilon>0$, $n\gg 1$, and
$0<\delta<\bar\delta_1$,
 \begin{align}
 \label{thmc-n-eq2}
u^*(t, x)
&\leq (U(nNT+t, 0)u_-)(x)+\epsilon/2 \quad \forall\,\, t\in [0,T],\,\, x\in\bar D.
\end{align}
By Theorem A, there is $0< \delta_1\le\bar\delta_1$  such that
for $\delta<\delta_1$,
\begin{align}
\label{thmc-n-eq3}
 (U(nNT+t, 0)u_-)(x)\leq (U^{\delta}(nNT+t, 0)u_-)(x)+\frac{\epsilon}{2}\quad \forall\,\, t\in [0,T],\,\, x\in\bar D.
\end{align}
By Proposition \ref{comparison-prop}  and \eqref{thmc-n-eq1}, we have
\begin{equation}
\label{thmc-n-eq4}
(U^{\delta}(nNT+t, 0)u_-)(x)=(U^{\delta}(t, 0)U^{\delta}(nNT, 0)u_-)(x)\leq(U^{\delta}(t, 0)u^*_{\delta}(t, \cdot))(x)=u_\delta^*(t,x)
\end{equation}
for  $t\in [0,T]$ and $x\in\bar D$. \eqref{thmc-n-eq0} then follows from \eqref{thmc-n-eq2}-\eqref{thmc-n-eq4}.

Next, we show that for any $\epsilon>0$, there is $\delta_2>0$ such that for $0<\delta<\delta_2$,
\begin{equation}
\label{thmc-n-eq6}
u^*(t, x)\geq u_{\delta}^*(t, x)-\epsilon\quad \forall\,\, t\in [0,T],\,\, x\in\bar D.
\end{equation}
Choose $M\gg 1$ such that $f(t,x,M)<0$ for $t\in\RR$ and $x\in \bar D$. Put
$$
u^+(x)=M\quad \forall\,\, x\in\bar D.
$$
Then for all $\delta>0$,
\begin{equation}
\label{thmc-n-eq7}
u_\delta^*(0,x)\le u^+(x)\quad \forall\,\, x\in\bar D.
\end{equation}
By Proposition \ref{time-periodic-solu}, there is $N\gg 1$ such that
\begin{equation}
\label{thmc-n-eq8}
u^*(t, x)\geq (U(NT+t, 0)u^+)(x)-\epsilon/2\quad \forall\,\, t\in [0,T],\,\, x\in\bar D.
\end{equation}
By Theorem A, there is $\delta_2>0$  such that for $0<\delta<\delta_2$,
\begin{equation}
\label{thmc-n-eq9}
(U(NT+t, 0)u^+)(x)\geq (U^{\delta}(NT+t, 0)u^+)(x)-\frac{\epsilon}{2}\quad \forall\,\, t\in [0,T],\,\, x\in\bar D.
\end{equation}
By \eqref{thmc-n-eq7},
\begin{equation}
\label{thmc-n-eq11}
 (U^{\delta}(NT+t, 0)u^+)(x)=(U^{\delta}(t, 0)U^{\delta}(NT, 0)u^+)(x)\geq (U^{\delta}(t, 0)u_{\delta}^*(t, \cdot))(x)=u_{\delta}^*(t, x)
\end{equation}
for  $t\in [0,T]$ and $x\in\bar D$. \eqref{thmc-n-eq6} then follows from \eqref{thmc-n-eq8}-\eqref{thmc-n-eq11}.

So, for any $\epsilon>0$, there exists $\delta_0=\min\{\delta_1, \delta_2\}$, such that for any $0<\delta<\delta_0$, we have
\[
|u^*(t, x)-u_{\delta}^*(t, x)|\leq \epsilon\,\ \text{ uniform in } t>0 \text{ and } x\in \bar D.
\]
\end{proof}

\begin{proof}[Proof of Theorem C in the periodic boundary condition case] We assume $D=\R^N$, and
$B_{r,b}u=B_{r,P}u$ in \eqref{main-random-eigenvalue}, and $B_{n,b}u=B_{n,P}u$ in
\eqref{main-nonlocal-eigenvalue}.
It can be proved by the similar arguments as in the Neumann boundary condition case.
\end{proof}

\noindent {\bf Acknowledgments.}  The authors would like to thank the referees for the valuable comments and suggestions which
improved the presentation of this paper considerably.

\end{document}